\documentclass[a4paper,11pt]{article}

\usepackage{float}
\usepackage{amsmath,amssymb,amsthm,amscd,algorithmic,algorithm}
\usepackage[mathscr]{eucal}
\usepackage[all]{xy}
\usepackage{graphicx}
\usepackage{appendix}
\usepackage{color}
\usepackage{hyperref}
\hypersetup{
colorlinks=true,
linkcolor=blue,
citecolor=blue,
urlcolor=blue
}

\makeatletter
    
    \@addtoreset{equation}{section}
  \makeatother

\newcommand{\plim}{\varprojlim}
\newcommand{\mcal}{\mathcal}
\newcommand{\mbf}{\mathbf}

\newcommand{\mbb}{\mathbb}
\newcommand{\mrm}{\mathrm}
\newcommand{\vphi}{\varphi}

\newtheorem{theorem}{Theorem}[section]
\newtheorem{corollary}[theorem]{Corollary}
\newtheorem{lemma}[theorem]{Lemma}
\newtheorem{proposition}[theorem]{Proposition}

\theoremstyle{definition}

\newtheorem{remark}[theorem]{Remark}

\newtheorem*{acknowledgments}{Acknowledgments}

\counterwithin{table}{section}
\counterwithin{algorithm}{section}

\title{Torsion of elliptic curves over $\mathbb{Q}_p$
with good reduction in cyclotomic extensions} 

\author{Yoshiyasu Ozeki\footnote{
Faculty of Science, Kanagawa University,
3-27-1 Rokkakubashi, Kanagawa-ku, Yokohama-shi, Kanagawa 221-8686, JAPAN
\endgraf
e-mail: {\tt ozeki@kanagawa-u.ac.jp}
}
and Manabu Yoshida\footnote{
Faculty of Mathematics Education, Osaka Kyoiku University, 
4-698-1 Asahigaoka, Kashiwara, Osaka 582-8582, JAPAN
\endgraf
email: {\tt yoshida-m95@cc.osaka-kyoiku.ac.jp}
\endgraf
}}

\begin{document}
\maketitle

\begin{abstract}
In this paper, for every prime $p$ and every $0\le n\le \infty$, 
we classify the structure of the torsion subgroup of the group of $\mbb{Q}_p(\mu_{p^n})$-rational points of elliptic curves 
over $\mbb{Q}_p$ with good reduction, 
where $\mu_{p^n}$ is the set of the $p^n$-th roots of unity. 

\end{abstract}

    \tableofcontents


\section{Introduction}

Let $p$ be a prime,
$K$ a finite extension of $\mbb{Q}_p$ 
and $A$ an abelian variety over $K$.
Let $L$ be an algebraic extension of $K$.
It is a theorem of Mattuck \cite{Mat55} that, if 
$L$ is a finite extension of $K$, 
then the Mordell-Weil group 
$A(L)$ is isomorphic to the direct sum of $\mbb{Z}_p^{\oplus \dim A\cdot [L:\mbb{Q}_p]}$ 
and a finite group.  
Furthermore, Clark and Xarles \cite{ClXa08} gave an explicit upper bound for  
the order of the torsion subgroup $A(L)_{\mrm{tor}}$ of $A(L)$
in terms of $p, \dim A$ and some numerical invariants of  $L$
when $A$ has anisotropic reduction (which includes the case of good reduction).
In the case where $L$ is of infinite degree,
Imai \cite{Ima75} showed that $A(K(\mu_{p^{\infty}}))_{\mrm{tor}}$ is finite
if $A$ has potentially good reduction.
Here,  $K(\mu_{p^{\infty}})$ denotes the extension field of 
$K$ obtained by adjoining   all $p$-power roots of unity.
The first author \cite{Oze24-2} studied an explicit order of $A(K(\mu_{p^{\infty}}))_{\mrm{tor}}$
when $A$ has complex multiplication.

The aim of this paper is to determine the possible group structure
of $E(\mbb{Q}_p)_{\mrm{tor}}$ and $E(\mbb{Q}_p(\mu_{p^{\infty}}))_{\mrm{tor}}$
for elliptic curves $E$ over $\mbb{Q}_p$ with good reduction.
Let us introduce some notation needed for our results.
We denote by $I$ the set of pairs $(m,k)$ of positive integers 
such that $m\mid p-1$ and 
$(\sqrt{p}-1)^2<km^2<(\sqrt{p}+1)^2$.
We also denote by $I_{\mrm{ord}}$ 
the subset of $I$ consisting of elements $(m,k)$
such that $km^2\not\equiv 1 \ \mrm{mod}\ p$.
Our first main result in this paper is as follows.
\begin{theorem} 
\label{E(Qp):good}
Let $E$ be an elliptic curve over $\mbb{Q}_p$
with good reduction.
\begin{itemize}
\item[{\rm (1)}]
Assume $p\ge 3$. 
Then, $E(\mbb{Q}_p)_{\mrm{tor}}$
is isomorphic to one of the following groups.
$$
\left\{
\begin{array}{ll}
 \mbb{Z}/m\mbb{Z}\times \mbb{Z}/mk\mbb{Z}, 
& (m,k)\in I,  \\
 0.
& 
\end{array}
\right.
$$
Each of these groups appears as $E(\mbb{Q}_p)_{\mrm{tor}}$ for 
some elliptic curve $E$ over $\mbb{Q}_p$ with good reduction.

Moreover, if $E$ has good ordinary reduction (resp.\ good supersigular reduciton), 
then $E(\mbb{Q}_p)_{\mrm{tor}}$
is isomorphic to one of the following groups in (I) (resp.\ (II)).
\begin{align*}
(I)&\left\{
\begin{array}{ll}
 \mbb{Z}/m\mbb{Z}\times \mbb{Z}/mk\mbb{Z}, 
& (m,k)\in I_{\mrm{ord}},  \\
 0.
& 
\end{array}
\right. \\
(II)&\ 
\left\{
\begin{array}{ll}
0, \mathbb{Z}/4\mathbb{Z}, \mathbb{Z}/7\mathbb{Z}, 
\mathbb{Z}/2\mathbb{Z}\times \mathbb{Z}/2\mathbb{Z}, 
& \mbox{if $p=3$},  \\
\mbb{Z}/(1+p)\mbb{Z},
& \mbox{if $p\equiv 1 \ \mrm{mod}\ 4$} \\ 
\mbb{Z}/(1+p)\mbb{Z}, 
\mbb{Z}/2\mbb{Z}\times \mbb{Z}/\frac{1+p}{2}\mbb{Z} 
& \mbox{if $p\not=3$, $p\equiv 3 \ \mrm{mod}\ 4$}. 
\end{array}
\right.
\end{align*}
Each of these groups in (I) (resp.\ (II)) 
appears as $E(\mbb{Q}_p)_{\mrm{tor}}$ for 
some elliptic curve $E$ over $\mbb{Q}_p$ with good  ordinary reduction
(resp.\ good supersingular reduction).
\item[{\rm (2)}] Assume $p=2$.
Then, $E(\mbb{Q}_2)_{\mrm{tor}}$
is isomorphic to one of the following groups.
$$
\left\{
\begin{array}{ll}
 \mbb{Z}/m\mbb{Z}, 
& m=1, 2,3,4,5,8,  \\
 \mbb{Z}/2\mbb{Z}\times \mbb{Z}/2k\mbb{Z}, 
& k=1,2.
\end{array}
\right.
$$
Each of these groups appears as $E(\mbb{Q}_2)_{\mrm{tor}}$ for 
some elliptic curve $E$ over $\mbb{Q}_2$ with good reduction.

Moreover, if $E$ has good ordinary reduction (resp.\ good supersigular reduciton), 
then $E(\mbb{Q}_2)_{\mrm{tor}}$
is isomorphic to one of the following groups in $(I)'$ (resp.\ $(II)'$).
\begin{align*}
(I)'&\ \left\{
\begin{array}{ll}
 \mbb{Z}/m\mbb{Z}, 
& m=2,4,8  \\
 \mbb{Z}/2\mbb{Z}\times \mbb{Z}/2k\mbb{Z},
& k=1,2\\ 
\end{array}
\right. \\
(II)'&\quad  0, \mbb{Z}/3\mbb{Z}, \mbb{Z}/5\mbb{Z}.
\end{align*}
Each of these groups in $(I)'$ (resp.\ $(II)'$) 
appears as $E(\mbb{Q}_2)_{\mrm{tor}}$ for 
some elliptic curve $E$ over $\mbb{Q}_2$ with good  ordinary reduction
(resp.\ good supersingular reduction).
\end{itemize}
\end{theorem}

Before stating our second result, 
it should be better to mention some known results 
on the group structures of elliptic curves over 
(infinite degree) abelian extensions of $\mbb{Q}$.
Chou \cite{Cho19} determined the possible torsion subgroups of $E(\mbb{Q}^{\mrm{ab}})$ 
for an elliptic curve $E$ over $\mbb{Q}$ 
and established the sharp bound $\# E(\mbb{Q}^{\mrm{ab}}) \leq 163$. 
Building on Chou's results, Gu\v{z}vi\'c and Vukorepa \cite{GuVu23} 
classified all possible torsion subgroups 
of $E(\mbb{Q}(\mu_{p^{\infty}}))$ in the case $p=2,3,5,7$ and $11$. 
where $\mu_{p^n}$ is the group of $p^n$-th roots of unity. 
We consider a $p$-adic analogue of these results. 
Let $E$ be an elliptic curve over $\mathbb{Q}_p$ with good reduction. 
As an analogue of Chou's result, it is natural to study the group structure of $E(\mathbb{Q}_p^{\mathrm{ab}})_{\mrm{tor}}$, 
where $\mbb{Q}_p^{\mrm{ab}}$ denotes the maximal abelian extensions of $\mbb{Q}_p$.
However, we immediately see that this group is always infinite 
(indeed, the reduction map induces an isomorphism between 
the prime-to-$p$ parts of $E(\mbb{Q}_p^{\mrm{ab}})_{\mrm{tor}}$
and that of $\bar{E}(\overline{\mbb{F}}_p)$, 
where $\bar{E}$ denotes the reduction of $E$ 
and $\overline{\mbb{F}}_p$ is the separable closure of $\mbb{F}_p$. 
So we study the torsion subgroup of $E(\mathbb{Q}_p(\mu_{p^{\infty}}))$,
which may be regarded as a $p$-adic analogue of the work of Gu\v{z}vi\'{c} and Vukorepa. 
The second main theorem below forms the central part of this paper. 

\begin{theorem} 
\label{E(Qp(mu)):good}
Let $E$ be an elliptic curve over $\mbb{Q}_p$.
\begin{itemize}
\item[{\rm (1)}]  If $E$ has good supersingular reduction, 
then it holds $E(\mbb{Q}_p(\mu_{p^{\infty}}))_{\mrm{tor}}
=E(\mbb{Q}_p)_{\mrm{tor}}$. 
(Thus the possible group structures of this group are given in Theorem 
\ref{E(Qp):good}.)
\item[{\rm (2)}] Assume $p\ge 3$
and assume also that $E$ has good ordinary reduction.
Then,  
\begin{itemize}
\item it holds 
$E(\mbb{Q}_p(\mu_{p^{\infty}}))_{\mrm{tor}}
=E(\mbb{Q}_p(\mu_p))_{\mrm{tor}}$. 
\item Moreover, 
$E(\mbb{Q}_p(\mu_{p^{\infty}}))_{\mrm{tor}}$ is isomorphic to one of the following groups in $(I)_{\infty}$.
\begin{align*}
(I)_{\infty}&\ \left\{
\begin{array}{ll}
 \mbb{Z}/m\mbb{Z}\times \mbb{Z}/mk\mbb{Z}, 
& (m,k)\in I_{\mathrm{ord}},  \\
 \mbb{Z}/p\mbb{Z}\times \mbb{Z}/p\mbb{Z},
& \\ 
 \mbb{Z}/p\mbb{Z}\times \mbb{Z}/2p\mbb{Z} 
& \mbox{if $p\le 5$}. 
\end{array}
\right.
\end{align*}
Each of these groups in $(I)_{\infty}$  
appears as $E(\mbb{Q}_p(\mu_{p^{\infty}}))_{\mrm{tor}}$ for 
some elliptic curve $E$ over $\mbb{Q}_p$ with good  ordinary reduction.
\end{itemize}
\item[{\rm (3)}]  Assume $p=2$ 
and assume also that $E$ has good ordinary reduction. 
Then, 
\begin{itemize}
\item it holds 
$E(\mbb{Q}_2(\mu_{2^{\infty}}))_{\mrm{tor}}
=E(\mbb{Q}_2(\mu_{8}))_{\mrm{tor}}$. 
\item Moreover, $E(\mbb{Q}_2(\mu_{2^{\infty}}))_{\mrm{tor}}$
is isomorphic to one of the following groups in $(I)'_{\infty}$.
\begin{align*}
(I)'_{\infty} &\ \left\{
\begin{array}{ll}
 \mbb{Z}/m\mbb{Z}, 
& m=4,8  \\
 \mbb{Z}/2\mbb{Z}\times \mbb{Z}/2k\mbb{Z},
& k=1,4\\ 
 \mbb{Z}/4\mbb{Z}\times \mbb{Z}/4\mbb{Z}. 
&  
\end{array}
\right.
\end{align*}
Each of these groups in $(I)'_{\infty}$  
appears as $E(\mbb{Q}_2(\mu_{2^{\infty}}))_{\mrm{tor}}$ for 
some elliptic curve $E$ over $\mbb{Q}_2$ with good  ordinary reduction.
\end{itemize}
\end{itemize}
\end{theorem}

As a consequence, we obtain explicit upper bounds 
for $E(\mbb{Q}_p)_{\mrm{tor}}$ and 
$E(\mbb{Q}_p(\mu_{p^{\infty}}))_{\mrm{tor}}$ when $E$ is an elliptic curve over $\mbb{Q}_p$ 
with good ordinary reduction (resp. good supersingular reduction) 
as stated in (I) (resp. (II)) below. 
All these bounds are sharp, 
where $\lfloor x \rfloor$ denotes the greatest integer less than or equal 
to $x$: 
\begin{align*}
(I) \hspace{5pt} \#E(\mbb{Q}_p)_{\mrm{tor}} &\leq \left\{
\begin{array}{ll}
\lfloor (\sqrt{p}+1)^2 \rfloor
& (p \geq 5),  \\
 6
& (p=3), \\
8 & (p=2).
\end{array}
\right.   
\#E(\mbb{Q}_p(\mu_{p^{\infty}}))_{\mrm{tor}} \leq \ \left\{
\begin{array}{ll}
 p^2 
& (p \geq 7),  \\
 2p^2
& (p=3,5), \\
16 & (p=2).
\end{array}
\right.
 \\
(II) \hspace{5pt} \#E(\mbb{Q}_p)_{\mrm{tor}}&= \#E(\mbb{Q}_p(\mu_{p^{\infty}}))_{\mrm{tor}} \leq 
\ \left\{ 
\begin{array}{ll}
1+p & (p \geq 5), \\
7 & (p=3),\\
5 & (p=2).
\end{array}
\right.
\end{align*}

Note that Theorem \ref{E(Qp):good} and Theorem \ref{E(Qp(mu)):good} above give the classifications 
of the possible group structures of 
$E(\mbb{Q}_p(\mu_{p^n}))_{\mrm{tor}}$ 
for all primes $p$ and $0\le n\le \infty$ 
except the case where $(p,n)=(2,2)$ and $E$ has good ordinary reduction. 
The classification result on the exceptional case 
is as follows.

\begin{theorem} 
\label{E(Q2(mu4)):good}
Let $E$ be an elliptic curve over $\mbb{Q}_2$
with good ordinary reduction. 
Then, $E(\mbb{Q}_2(\mu_{4}))_{\mrm{tor}}$
is isomorphic to one of the following groups in $(I)'_2$ .
\begin{align*}
(I)'_2 &\ \left\{
\begin{array}{ll}
 \mbb{Z}/4\mbb{Z}, 
& \\
 \mbb{Z}/2\mbb{Z}\times \mbb{Z}/2k\mbb{Z},
& k=1,2,4\\ 
 \mbb{Z}/4\mbb{Z}\times \mbb{Z}/4\mbb{Z}. 
&  
\end{array}
\right.
\end{align*}
Each of these groups in $(I)'_2$ 
appears as $E(\mbb{Q}_2(\mu_{4}))_{\mrm{tor}}$ for 
some elliptic curve $E$ over $\mbb{Q}_2$ with good  ordinary reduction.
\end{theorem}

Therefore,  we conclude that,
for all primes $p$ and $0\le n\le \infty$, 
we obtained the complete classifications 
of the  groups those arise as  
$E(\mbb{Q}_p(\mu_{p^n}))_{\mrm{tor}}$ 
for some elliptic curve $E$ over $\mbb{Q}_p$ with 
good reduction. 

\begin{corollary}
Assume $p\ge 3$.
Let $K_{\infty}$ be the cyclotomic 
$\mbb{Z}_p$-extension of $\mbb{Q}_p$.
Then,  we have
$E(K_{\infty})_{\mrm{tor}}=E(\mbb{Q}_p)_{\mrm{tor}}$ 
for an elliptic curve $E$ over $\mbb{Q}_p$
with good reduction.
\end{corollary}

\begin{proof}
If $E$ has good supersingular reduction,
the result is clear by Theorem \ref{E(Qp(mu)):good} (1).
In the case where  $E$ has good ordinary reduction,
the result follows from Theorem \ref{E(Qp(mu)):good} (2); 
$E(K_{\infty})_{\mrm{tor}}
= E(K_{\infty})_{\mrm{tor}}\cap
E(\mbb{Q}_p(\mu_{p^{\infty}}))_{\mrm{tor}}
=E(K_{\infty})_{\mrm{tor}}\cap
E(\mbb{Q}_p(\mu_{p}))_{\mrm{tor}}
=E(\mbb{Q}_p)_{\mrm{tor}}$.
\end{proof}

The organization of the paper is as follows.
In Section \ref{Section:E(Qp)}, we give a 
proof of Theorem \ref{E(Qp):good}.
The arguments differs significantly depending on whether $p$ is odd
or $p=2$.
When $p$ is odd, we use theoretical arguments involving the theory of canonical lifts. 
In case $p=2$, by using MAGMA \cite{BoCa06} and Algorithm \ref{alg:torsion}, 
we explicitly find elliptic curves listed in the Cremona database 
with prescribed torsion subgroups. 
Section \ref{Section:E(Qp(mu))} is the main part of this paper.
In this seciton, we give 
proofs of Theorem \ref{E(Qp(mu)):good}
and Theorem 
\ref{E(Q2(mu4)):good}.
As in Section \ref{Section:E(Qp)}, 
the arguments also differ depending on whether  $p$ is odd
or $p=2$.
For the case where $p$ is odd,
the key is Proposition \ref{Main:cycl}, 
which gives a classification of $p$-parts of the torsion subgroup of 
$E(\mbb{Q}_p(\mu_{p^{\infty}}))$.
For the case where $p=2$, 
theoretical perspectives such as ramification theory 
play an even more important role in addition to verification 
using the Cremona database and computations by MAGMA. 
In Appendix \ref{Data}, we provide data on certain 
extensions of $\mbb{Q}_2$ and some elliptic curves that are 
required for our proof.

\vspace{5mm}
\noindent
{\bf Notation :}
In this paper, $p$-adic fields are finite extension fields of $\mbb{Q}_p$.
If $F$ is an algebraic extension of $\mbb{Q}_p$, 
we denote by $G_F$  the absolute Galois group 
$\mrm{Gal}(\overline{\mbb{Q}}_p/F)$ of $F$.
We also denote by $\mu_{p^n}$ the set of $p^n$-th roots of unity in 
$\overline{\mbb{Q}}_p$ and 
$\mu_{p^{\infty}}:=\cup_{m\ge 0} \mu_{p^m}$. 
The labels of elliptic curves in this paper follow 
the convention used in the Cremona database. 
The data available in the LMFDB \cite{lmfdb} is also useful for referencing elliptic curves; 
however, note that the labeling in the LMFDB differs from 
that of the Cremona label.


%
%

\section{Group structures of $E(\mbb{Q}_p)_{\mrm{tor}}$}
\label{Section:E(Qp)}

The aim of this section is to prove Theorem
\ref{E(Qp):good},
which gives the complete list of 
the groups those arise as the torsion subgroups of the Mordell-Weil groups of elliptic curves over $\mbb{Q}_p$ with good reduction.
Theorem
\ref{E(Qp):good}
is a combination of Theorem \ref{E(Qp):ord} 
and Theorem \ref{E(Qp):ord:p=2} below.

In the rest of this paper, we use the following notations.
\begin{itemize}
\item For an elliptic curve $E$ over $\mbb{Q}_p$
with good reduction, we denote by $\hat{E}$ and $\bar{E}$ the formal group over $\mbb{Z}_p$ associated with $E$ and 
the reduction of $E$, respectively.
We denote by $T_p(E):=\plim_{n}E[p^n]$ the $p$-adic Tate module 
of $E$ and put $V_p(E)=T_p(E)\otimes_{\mbb{Z}_p} \mbb{Q}_p$.
Similarly, we often use notations 
$T_p(\hat{E}), V_p(\hat{E}), T_p(\bar{E})$ and $V_p(\bar{E})$.

\item 
For a field $K$, we denote by 
$\mcal{E}(K)$ the set of the 
isomorphism classes of groups 
which are isomorphic to the torsion subgroup  $E(K)_{\mrm{tor}}$
of $E(K)$ for some elliptic curve $E$ over $K$.
\item For integers $m,k\ge 1$, we set 
$G_{m,k}:=\mbb{Z}/m\mbb{Z}\times \mbb{Z}/mk\mbb{Z}$.
\item We denote by $I$ the set of pairs $(k,m)$ of positive integers 
such that  $m\mid p-1$ and 
$(\sqrt{p}-1)^2<km^2<(\sqrt{p}+1)^2$.
\end{itemize}
The set $\mcal{E}(\mbb{F}_p)$ was well-studied 
by Hasse, Deuring,..., R\"uck and Volock.
The following statement is due to \cite[Lemma 3.5]{BPS12}.
\begin{theorem}
\label{E(Fp)}
$\mcal{E}(\mbb{F}_p)=
\left\{G_{m,k}\mid (m,k)\in I\right\}.$
\end{theorem}

We denote by 
$\mcal{E}_{\mrm{good}}(\mbb{Q}_p)$
(resp.\ $\mcal{E}_{\mrm{ord}}(\mbb{Q}_p)$, 
resp.\ $\mcal{E}_{\mrm{ss}}(\mbb{Q}_p)$)
the subset of $\mcal{E}(\mbb{Q}_p)$ 
consisting of 
isomorphism classes of $E(\mbb{Q}_p)_{\mrm{tor}}$ 
for some elliptic curve $E$ over $\mbb{Q}_p$ with good reduction 
(resp.\ good ordinary reduction,
resp.\ good supersingular reduction).
We clearly have 
$$
\mcal{E}_{\mrm{good}}(\mbb{Q}_p)=
\mcal{E}_{\mrm{ord}}(\mbb{Q}_p)\cup
\mcal{E}_{\mrm{ss}}(\mbb{Q}_p).
$$ 
Let $E$ be an elliptic curve over $\mbb{Q}_p$
with good reduction.
We have an exact sequence 
$$
0\to \hat{E}(\mbb{Q}_p)\to E(\mbb{Q}_p)\to \bar{E}(\mbb{F}_p)\to 0
$$
of modules
(cf.\ \cite[Section VII.2]{Sil09})\footnote{
In this paper, for an algebraic extension $K$ of $\mbb{Q}_p$ with maximal ideal $\mbf{m}_K$,
we denote by $\hat{E}(K)$ the group $\hat{E}(\mbf{m}_K)=\mbf{m}_K$ 
determined by the formal group $\hat{E}$ (cf.\ \cite[Chapter IV.3]{Sil09}).}.
Since the pro-$p$ group 
$\hat{E}(\mbb{Q}_p)$ has no torsion points if $p\ge 3$ 
(cf.\ \cite[Section IV, Proposition 3.2 and Theorem 6.1]{Sil09})
and $\bar{E}(\mbb{F}_p)[p^{\infty}]=\bar{E}(\mbb{F}_p)[p]$ by the Hasse bound,
the reduction map induces an isomorphism 
\begin{equation}
\label{p'-part}
E(\mbb{Q}_p)_{p'}\simeq \bar{E}(\mbb{F}_p)_{p'}
\end{equation}
and an injection 
\begin{equation}
\label{p-part}
E(\mbb{Q}_p)[p^{\infty}]
\hookrightarrow \bar{E}(\mbb{F}_p)[p]\quad \mbox{if $p\ge 3$}.
\end{equation}
Here, for a module $M$, 
we denote by 
$M[p^n]$   
the submodule of $M$ killed by $p^n$, 
$M[p^{\infty}]:=\cup_{n>0}M[p^n]$, and also denote by $M_{p'}$ 
the prime-to-$p$ part of $M$.

\subsection{The case $p\ge 3$}

We study $\mcal{E}_{\mrm{ord}}(\mbb{Q}_p)$,
$\mcal{E}_{\mrm{ss}}(\mbb{Q}_p)$ and 
$\mcal{E}_{\mrm{good}}(\mbb{Q}_p)$ for an odd prime $p$.
We use the following notations.
\begin{itemize}
\item We denote by $I_{\mrm{ord}}\, (\subset I)$ 
the set of pairs $(k,m)$ of positive integers 
such that $m\mid p-1$,  
$(\sqrt{p}-1)^2<km^2<(\sqrt{p}+1)^2$ 
and $km^2\not\equiv 1\ \mrm{mod}\ p$.
\item We denote by $I_{\mrm{ss}}\, (\subset I)$ 
the set of pairs $(k,m)$ of positive integers 
such that $m\mid p-1$,  
$(\sqrt{p}-1)^2<km^2<(\sqrt{p}+1)^2$ 
and $km^2\equiv 1\ \mrm{mod}\ p$.
\end{itemize}
By definition we have $I=I_{\mrm{ord}}\cup I_{\mrm{ss}}$.
A straight forward calculation shows that 
the set $I_{\mrm{ss}}$ coincides with 
$\{(1,1), (1,4), (1,7), (2,1) \}$
(resp.\ $\{ (1,1+p) \}$, resp.\ $\{(1,1+p), (2,\frac{1+p}{4}) \}$)
if $p=3$ (resp.\ $p \equiv 1 \ \mrm{mod} \ 4$, 
resp.\ $p\neq 3$ and $p \equiv 3 \ \mrm{mod} \ 4$).

\begin{theorem}[=Theorem \ref{E(Qp):good} (1)] 
\label{E(Qp):ord}
Assume $p\ge 3$. 
\begin{itemize}
\item[{\rm (1)}] $\mcal{E}_{\mrm{ord}}(\mbb{Q}_p)
=\{ G_{m,k} \mid \mbox{$(m,k)\in I_{\mrm{ord}}$}\}
\cup \{0\}$.
\item[{\rm (2)}]
$\mcal{E}_{\mrm{ss}}(\mbb{Q}_p)
= \{G_{m,k}\mid (m,k)\in I_{\mrm{ss}} \}$.
Explicitly, we have 
\begin{align*}
\mcal{E}_{\mrm{ss}}(\mbb{Q}_p) 
&= \left\{ \, 
\begin{aligned}
&\{G_{1,1}, G_{1,4}, G_{1,7}, G_{2,1}\} & if & \ p=3\\
&\{G_{1,1+p}\} & if & \ p \equiv 1 \ \mrm{mod} \ 4\\
&\{G_{1,1+p}, G_{2,\frac{1+p}{4}} \} & if & \ p\neq 3,\ p \equiv 3 \ \mrm{mod} \ 4\\
\end{aligned}
\right.
\end{align*}
\item[{\rm (3)}]
$\mcal{E}_{\mrm{good}}(\mbb{Q}_p)
 = \mcal{E}(\mbb{F}_p)\cup \{0\}
= \{G_{m,k}\mid (m,k)\in I \} \cup \{0\}$.
\end{itemize}
\end{theorem}
For the proof of the theorem,
it suffices to show (1) and (2).

\subsubsection{The case of ordinary reduction}
\label{Qp:ord:odd}
We show Theorem \ref{E(Qp):ord} (1).
Let $E$ be an elliptic curve over $\mbb{Q}_p$ with good reduction.
Suppose that $\bar{E}$ is ordinary and 
$\bar{E}(\mbb{F}_p)=G_{m,k}$. 
Since  $a_p(E):=1+p-\# \bar{E}(\mbb{F}_p)$ is prime to $p$,
we have $km^2\not \equiv 1\ \mrm{mod}\ p$.
By \eqref{p'-part} and \eqref{p-part},
we see that 
$E(\mbb{Q}_p)_{\mrm{tor}}$
is isomorphic to $G_{m, k}$
or 
$G_{m, k/p}$ (with $p\mid k$).
Hence we obtain 
\begin{equation*}
\mcal{E}_{\mrm{ord}}(\mbb{Q}_p)
\subset  
\mcal{E}^1_{\mrm{ord}}(\mbb{Q}_p)
\cup 
\mcal{E}^2_{\mrm{ord}}(\mbb{Q}_p)
\cup 
\mcal{E}^3_{\mrm{ord}}(\mbb{Q}_p)
\end{equation*}
where 
\begin{align*}
\mcal{E}^1_{\mrm{ord}}(\mbb{Q}_p) & :=
\{ G_{m,k} \mid \mbox{$(m,k)\in I_{\mrm{ord}}$, $p\nmid k$}\}, \\
\mcal{E}^2_{\mrm{ord}}(\mbb{Q}_p) & :=
\{ G_{m,k} \mid \mbox{$(m,k)\in I_{\mrm{ord}}$, $p\mid k$}\}, \\
\mcal{E}^3_{\mrm{ord}}(\mbb{Q}_p)& :=
\{ G_{m,k/p} \mid \mbox{$(m,k)\in I_{\mrm{ord}}$, $p\mid k$}\}.
\end{align*}
In the following, we show that the inclusion ``$\subset$"
above is in fact equal.
It suffices to show that each $\mcal{E}^i_{\mrm{ord}}(\mbb{Q}_p)$
is a subset of $\mcal{E}_{\mrm{ord}}(\mbb{Q}_p)$.

Let $G_{m,k}\in \mcal{E}^1_{\mrm{ord}}(\mbb{Q}_p)$.
By Theorem \ref{E(Fp)},
there exists an elliptic curve $\bar{E}$ over $\mbb{F}_p$
such that $G_{m,k}\simeq \bar{E}(\mbb{F}_p)$.
By considering lifts to $\mbb{Z}_p$ of the coefficients of 
the Weierstrass equation of $\bar{E}$, 
we obtain an elliptic curve $E$ over $\mbb{Q}_p$
with good ordinary reduction whose reduction is $\bar{E}$.
Since $\bar{E}(\mbb{F}_p)$ has no $p$-torsion points  
by $p\nmid k$, 
it follows from \eqref{p'-part} and \eqref{p-part}
that $E(\mbb{Q}_p)_{\mrm{tor}}\simeq G_{m,k}$.
Thus we have 
$\mcal{E}^1_{\mrm{ord}}(\mbb{Q}_p)\subset \mcal{E}_{\mrm{ord}}(\mbb{Q}_p)$.

Let $G_{m,k}\in \mcal{E}^2_{\mrm{ord}}(\mbb{Q}_p)$.
In this case, we see $(m,k)\in \{(1,p), (1,2p)\}$ if $p\le 5$
and $(m,k)\in \{(1,p)\}$ if $p> 5$.
Write $k=jp$.
By Theorem \ref{E(Fp)}, 
there exists an elliptic curve $\bar{E}$ over $\mbb{F}_p$
such that $\bar{E}(\mbb{F}_p)\simeq G_{m,k}=G_{1,jp}$.
Let $E_0/_{\mbb{Q}_p}$ be the canonical lift of $\bar{E}$.
Since $\mrm{End}_{\mbb{Q}_p} E_0=\mrm{End}_{\mbb{F}_p}\bar{E}$ 
is an order of an imaginary quadratic field,
the $G_{\mbb{Q}_p}$-action on $E_0[p]$ 
factors an abelian quotient.
In addition, 
we see $\hat{E}[p]\simeq \mbb{F}_p(1)$ and 
$\bar{E}[p]\simeq \mbb{F}_p$
as $G_{\mbb{Q}_p}$-modules
since $\bar{E}(\mbb{F}_p)[p]$ is not trivial.
Hence we have (non-canonical) isomorphisms 
$E_0[p]\simeq \hat{E}[p]\oplus \bar{E}[p]
\simeq \mbb{F}_p(1)\oplus \mbb{F}_p$
of $G_{\mbb{Q}_p}$-modules.
Thus we obtain $E_0(\mbb{Q}_p)[p]\simeq \bar{E}(\mbb{F}_p)[p]\simeq G_{1,p}$,
which gives $E(\mbb{Q}_p)[p^{\infty}]\simeq  G_{1,p}$
by \eqref{p-part}.
On the other hand, we also have $E_0(\mbb{Q}_p)_{p'}\simeq \bar{E}(\mbb{F}_p)_{p'}\simeq G_{1,j}$ by \eqref{p'-part}.
Hence we obtain $E_0(\mbb{Q}_p)_{\mrm{tor}}\simeq G_{1,jp}$.
Therefore, we obtain 
$\mcal{E}^2_{\mrm{ord}}(\mbb{Q}_p)\subset \mcal{E}_{\mrm{ord}}(\mbb{Q}_p)$.

Let $G_{m,k/p}\in \mcal{E}^3_{\mrm{ord}}(\mbb{Q}_p)$.
In this case, we see
 $(m,k)\in \{(1,p), (1,2p)\}$ if $p\le 5$
and $(m,k)\in \{(1,p)\}$ if $p> 5$.
Since $G_{1,2}\in \mcal{E}^1_{\mrm{ord}}(\mbb{Q}_p)$ if $p\le 5$, 
it suffices to show  that $G_{1,1}\, (=0)$ is an element of $\mcal{E}_{\mrm{ord}}(\mbb{Q}_p)$.
By Theorem \ref{E(Fp)}, 
there exists an elliptic curve $\bar{E}$ over $\mbb{F}_p$
such that $\bar{E}(\mbb{F}_p)\simeq G_{m,k}=G_{1,p}$.
Take any lift $\tilde{E}$  to $W_2:=W(\mbb{F}_p)/p^2W(\mbb{F}_p)$ 
of $\bar{E}$
such that 
$\dim_{\mbb{F}_p} \tilde{E}(W_2)\otimes_{\mbb{Z}} \mbb{F}_p = 1$ 
(there exist $p-1$ isomorphism class of such $\tilde{E}$ 
by \cite[Lemma 3.1 and Lemma 3.2]{DaWe08}).
Taking any lift $E/_{\mbb{Q}_p}$ of $\tilde{E}$,
we have $E(\mbb{Q}_p)[p]=0$ by \cite[Lemma 3.1]{DaWe08}.
Since we have $E(\mbb{Q}_p)_{p'}\simeq \bar{E}(\mbb{F}_p)_{p'}=0$,
it holds $E(\mbb{Q}_p)_{\mrm{tor}}=0=G_{1,1}$.
Consequently, we have 
$\mcal{E}^3_{\mrm{ord}}(\mbb{Q}_p)\subset \mcal{E}_{\mrm{ord}}(\mbb{Q}_p)$.
Thus we proved Theorem \ref{E(Qp):ord} (1).

\subsubsection{The case of supersingular reduction}
We show Theorem \ref{E(Qp):ord} (2).
Let $E$ be an elliptic curve over $\mbb{Q}_p$ with good reduction.
Suppose that $\bar{E}$ is supersingular. 
By \eqref{p'-part} and \eqref{p-part},
we have $E(\mbb{Q}_p)_{\mrm{tor}}\simeq \bar{E}(\mbb{F}_p)$.
Thus we have 
\begin{equation}
\label{Ess}
\mcal{E}_{\mrm{ss}}(\mbb{Q}_p)\subset 
\{G_{m,k}\mid (m,k)\in I_{\mrm{ss}}\}.
\end{equation}
Conversely, let $G_{m,k}$  be in the right hand side of the above
and let $\bar{E}$ be the elliptic curve over $\mbb{F}_p$ with
$\bar{E}(\mbb{F}_p)\simeq G_{m,k}$. 
Note that $\bar{E}$ is supersingular since $a_p(E)=1+p-\# \bar{E}(\mbb{F}_p)\equiv 0\ \mrm{mod}\ p$.
By taking any lift $E/_{\mbb{Q}_p}$ of $\bar{E}$,
we have $E(\mbb{Q}_p)_{\mrm{tor}}\simeq \bar{E}(\mbb{F}_p)\simeq G_{m,k}$
by \eqref{p'-part} and \eqref{p-part} again.
Thus  the inclusion ``$\subset$" in \eqref{Ess} is in fact equal.
This finises a proof of Theorem \ref{E(Qp):ord} (2).

\subsection{The case $p=2$}
We study $\mcal{E}_{\mrm{ord}}(\mbb{Q}_p)$,
$\mcal{E}_{\mrm{ss}}(\mbb{Q}_p)$ and 
$\mcal{E}_{\mrm{good}}(\mbb{Q}_p)$ for $p=2$.

\begin{theorem}[=Theorem \ref{E(Qp):good} (2)]
\label{E(Qp):ord:p=2} 
\begin{itemize}
\item[{\rm (1)}] 
$\mcal{E}_{\mrm{ord}}(\mbb{Q}_2)
=\{ G_{1,2}, G_{1,4}, G_{1,8}, G_{2,1}, G_{2,2}\}$.
\item[{\rm (2)}]
$\mcal{E}_{\mrm{ss}}(\mbb{Q}_2)
= \{ G_{1,1}, G_{1,3},G_{1,5} \}$.
\item[{\rm (3)}]
$\mcal{E}_{\mrm{good}}(\mbb{Q}_2)
= \{G_{1,1}, G_{1,2}, G_{1,3}, G_{1,4}, G_{1,5}, G_{1,8}, G_{2,1}, G_{2,2}\} $.
\end{itemize}
\end{theorem}
For the proof of the theorem,
it suffices to show (1) and (2).

\subsubsection{The case of ordinary reduction}
We show Theorem \ref{E(Qp):ord:p=2} (1).
Let $E$ be an elliptic curve over $\mbb{Q}_2$
with good reduction.
Suppose that $\bar{E}$ is ordinary. 
Then torsion subgroup of $\hat{E}(\mbb{Q}_2)$ have order at most 2 by \cite[Chapter IV, Proposition 3.2 and Theorem 6.1]{Sil09} 
and those of  $\bar{E}(\mbb{F}_2)$ have order at most 5 by the Hasse bound.
Moreover, since $\hat{E}[2]$ and $\bar{E}[2]$ are cyclic of order 2, 
$G_{\mbb{Q}_2}$ acts trivially on them.
Thus  $\hat{E}(\mbb{Q}_2)_{\mrm{tor}}\simeq \mbb{Z}/2\mbb{Z}$ 
and $\bar{E}(\mbb{F}_2)$ is isomorphic to either 
$\mbb{Z}/2\mbb{Z}$ or $\mbb{Z}/4\mbb{Z}$.  
Since we have an exact sequence
$0\to \hat{E}(\mbb{Q}_2)_{\mrm{tor}}\to 
E(\mbb{Q}_2)_{\mrm{tor}} \to \bar{E}(\mbb{F}_2)$
of modules,
the group 
$E(\mbb{Q}_2)_{\mrm{tor}}$ is isomorphic to one of the following groups:
$$
G_{1,2}, G_{1,4}, G_{1,8}, G_{2,1}, G_{2,2}.
$$
In fact, it follows from MAGMA calculation with Algorithm \ref{alg:torsion} 
that, for each  $G_{m,k}$ appearing above, 
there exists an elliptic curve $E$ over $\mathbb{Q}_2$ 
with good ordinary reduction 
such that $E(\mathbb{Q}_2)_{\mrm{tor}}$ 
is isomorphic to $G_{m,k}$. For example, 
\begin{itemize}
\item 
$E=$\href{https://www.lmfdb.org/EllipticCurve/Q/15/a/1}{15.a5} satisfies $E(\mbb{Q}_2)_{\mrm{tor}}\simeq G_{1,2}$,
\item 
$E=$\href{https://www.lmfdb.org/EllipticCurve/Q/15/a/4}{15.a7} satisfies $E(\mbb{Q}_2)_{\mrm{tor}}\simeq G_{1,4}$,
\item 
$E=$\href{https://www.lmfdb.org/EllipticCurve/Q/15/a/8}{15.a4} satisfies $E(\mbb{Q}_2)_{\mrm{tor}}\simeq G_{1,8}$,
\item 
$E=$\href{https://www.lmfdb.org/EllipticCurve/Q/15/a/2}{15.a2} satisfies $E(\mbb{Q}_2)_{\mrm{tor}}\simeq G_{2,1}$,
\item 
$E=$\href{https://www.lmfdb.org/EllipticCurve/Q/15/a/5}{15.a1} satisfies $E(\mbb{Q}_2)_{\mrm{tor}}\simeq G_{2,2}$.
\end{itemize}
This finishes a proof of Theorem \ref{E(Qp):ord:p=2} (1).

\subsubsection{The case of supersingular reduction}
We show Theorem \ref{E(Qp):ord:p=2} (2).
Let $E$ be an elliptic curve over $\mbb{Q}_2$
with good reduction.
Suppose that $\bar{E}$ is supersingular. 
Note that $E(\mbb{Q}_2)[2]=0$ since $E[2]$
is irreducible as a $G_{\mbb{Q}_2}$-module
(cf.\ \cite[Section 1.12, Proposition 12]{Ser72}).
Hence, it follows from \eqref{p'-part} that 
the reduction map induces an isomorphism
$E(\mbb{Q}_2)_{\mrm{tor}}\simeq \bar{E}(\mbb{F}_2)$.
By the Hasse bound, we have 
$\#\bar{E}(\mbb{F}_2)\in \{1,3,5\}$.
Thus, $E(\mbb{Q}_2)_{\mrm{tor}}$ is isomorphic to 
one of $G_{1,1}$, $G_{1,3}$ and $G_{1,5}$.
In fact, it follows from MAGMA calculation with Algorithm \ref{alg:torsion} 
that, for each  $G_{m,k}$ appearing above, 
there exists an elliptic curve $E$ over $\mathbb{Q}_2$ 
with good supersingular reduction 
such that $E(\mathbb{Q}_2)_{\mrm{tor}}$ 
is isomorphic to $G_{m,k}$. For example, 
\begin{itemize}
\item 
$E=$\href{https://www.lmfdb.org/EllipticCurve/Q/67/a/1}{67.a1} satisfies $E(\mbb{Q}_2)_{\mrm{tor}}\simeq G_{1,1}$,
\item 
$E=$\href{https://www.lmfdb.org/EllipticCurve/Q/19/a/2}{19.a1} satisfies $E(\mbb{Q}_2)_{\mrm{tor}}\simeq G_{1,3}$,
\item 
$E=$\href{https://www.lmfdb.org/EllipticCurve/Q/11/a/2}{11.a1} satisfies $E(\mbb{Q}_2)_{\mrm{tor}}\simeq G_{1,5}$.
\end{itemize}
This finishes a proof of Theorem \ref{E(Qp):ord:p=2} (2).

\section{Group structures of 
$E(\mbb{Q}_p(\mu_{p^n}))_{\mrm{tor}}$}
\label{Section:E(Qp(mu))}

The aim of this section is to prove Theorem 
\ref{E(Qp(mu)):good} and Theorem \ref{E(Q2(mu4)):good},
which gives the classification of the groups 
appearing as $E(\mbb{Q}_p(\mu_{p^{\infty}}))_{\mrm{tor}}$
for some elliptic curve $E$ over $\mbb{Q}_p$
with good reduction.
We begin with a proof of Theorem 
\ref{E(Qp(mu)):good} (1).

\begin{proof}[Proof of Theorem \ref{E(Qp(mu)):good} (1)]
By \cite[Section 1.12, Proposition 12]{Ser72},
we know that $E[p]$ is irreducible as $G_{\mbb{Q}_p}$-modules.
If we assume that $E(\mbb{Q}_p(\mu_{p^{\infty}}))[p]$ is not zero,
then it follows from the irreducibility that we have $E(\mbb{Q}_p(\mu_{p^{\infty}}))[p]=E[p]$.
This shows that $\mbb{Q}_p(E[p])$ is a subfield of 
$\mbb{Q}_p(\mu_{p^{\infty}})$.
Thus the prime-to-$p$ part of the ramification index of 
$\mbb{Q}_p(E[p])/\mbb{Q}_p$ must be a divisor of $p-1$.
However, by \cite[Section 1.12, Proposition 12]{Ser72} again, 
the ramification index of $\mbb{Q}_p(E[p])/\mbb{Q}_p$ is $p^2-1$.
This is a contradiction. Hence we obtain $E(\mbb{Q}_p(\mu_{p^{\infty}}))[p]=0$. 
In particular, we have 
$E(\mbb{Q}_p(\mu_{p^{\infty}}))[p^{\infty}]
=E(\mbb{Q}_p)[p^{\infty}]\, (=0)$.
On the other hand, it follows from the N\'eron-Ogg-Shafarevich criterion 
that the prime-to-$p$ parts of 
$E(\mbb{Q}_p(\mu_{p^{\infty}}))_{\mrm{tor}}$
 and $E(\mbb{Q}_p)_{\mrm{tor}}$
 coincides with each other.
Thus we conclude 
$E(\mbb{Q}_p(\mu_{p^{\infty}}))_{\mrm{tor}}=E(\mbb{Q}_p)_{\mrm{tor}}$
as desired. 
\end{proof}

For the arguments below, We use the following lemma. 

\begin{lemma}
\label{lem:red:def}
Let $E_{/\mbb{Q}_p}$ be an elliptic curve with good ordinary reduction. 
Let $\alpha$ be the non-unit root of $T^2-a_p(E)T+p=0$
and denote by $\chi_{\alpha}\colon G_{\mbb{Q}_p}\to \mbb{Z}_p^{\times}$
 the Lubin-Tate character\footnote{For the definition of Lubin-Tate characters, 
 see Appendix A.4 of Chapter III of \cite{Ser98}} 
 associated with $\alpha$.
Then, the $G_{\mbb{Q}_p}$-action on the $p$-adic Tate module $V_p(\hat{E})$ of $\hat{E}$
is given by $\chi_{\alpha}$.
\end{lemma}

\begin{proof}
Let $\chi\colon G_{\mbb{Q}_p}\to \mbb{Z}_p^{\times}$ be the character obtained by the 
$G_{\mbb{Q}_p}$-action on $V_p(\hat{E})$
and $\phi\colon G_{\mbb{Q}_p}\to \mbb{Z}_p^{\times}$ be the character obtained by the 
$G_{\mbb{Q}_p}$-action on $V_p(\bar{E})$.
Put $T^2-a_p(E)T+p=f_E(T)$.
For any crystalline $\mbb{Q}_p$-representation $V$ of $G_{\mbb{Q}_p}$,
let $D_{\mrm{cris}}(V)=(B_{\mrm{cirs}}\otimes_{\mbb{Q}_p} V)^{G_{\mbb{Q}_p}}$ 
be the Fontaine's filtered $\vphi$-module\footnote{For the basic notion of 
{\it p}-adic Hodge theory, it is helpful for the reader to refer \cite{Fon94a} and \cite{Fon94b}.}.
By $p$-adic Hodge theory, it is known that $f_E(T)$ coincides with the characteristic polynomial of 
the $\vphi$-module  $D_{\mrm{cris}}(V_p(E)^{\vee})$, that is, 
$f_E(T)=\mrm{det}(T-\vphi\mid D_{\mrm{cris}}(V_p(E)^{\vee}))$. 
Here, $\vee$ stands for the dual representation.
Moreover, this coincides with the products of  the characteristic polynomials of 
$D_{\mrm{cris}}(\mbb{Q}_p(\chi^{-1}))$
and $D_{\mrm{cris}}(\mbb{Q}_p(\phi^{-1}))$.
Since $\chi$ restricted to the inertia $I_{\mbb{Q}_p}$ coincides with the $p$-adic cyclotomic character,
for any choice of a uniformizer $\pi$ of $\mbb{Q}_p$, 
it follows from \cite[Proposition B.4]{Con11} that 
$\mrm{det}(T-\vphi\mid D_{\mrm{cris}}(\mbb{Q}_p(\chi^{-1})))
=\chi(\pi)\cdot \pi$,
which is independent of the choice of $\pi$
(here, we regard $\chi$ as a character of $\mbb{Q}_p^{\times}$
via the local reciprosity map).
Since $\chi(\pi)\cdot \pi$ has a postive $p$-adic valuation, we have 
$\chi(\pi)\cdot \pi=\alpha$ for any $\pi$.
By choosing  $\alpha$ as $\pi$, 
we have $\chi(\alpha)=1$. 
Since we have $\chi=\chi_{\alpha}$ on $I_{\mbb{Q}_p}$,
we find $\chi=\chi_{\alpha}$.
\end{proof}

\subsection{The case $p\ge 3$ }

We show Theorem \ref{E(Qp(mu)):good} (2).

\begin{lemma}
\label{p-bound}
Assume $p\ge 3$. 
Let $E$ be an elliptic curve over $\mbb{Q}_p$ with good ordinary reduction
and $\hat{E}$ the formal group associated with $E$.
Then, it holds 
$\# E(\mbb{Q}_p(\mu_{p^{\infty}}))[p^{\infty}]\le p^2$
and 
$\# \hat{E}(\mbb{Q}_p(\mu_{p^{\infty}}))[p^{\infty}]\le p$.
\end{lemma}
\begin{proof}
Consider an exact sequence
$$
0\to \hat{E}(\mbb{Q}_p(\mu_{p^{\infty}}))[p^{\infty}] 
\to E(\mbb{Q}_p(\mu_{p^{\infty}}))[p^{\infty}]
\to \bar{E}(\mbb{F}_p)[p^{\infty}]
$$
of $G_{\mbb{Q}_p}$-modules.
By the Hasse bound,
the order of $\bar{E}(\mbb{F}_p)[p^{\infty}]$
is at most $p$ (note that $p$ is now odd).
Thus it suffices to check that any element of 
$\hat{E}(\mbb{Q}_p(\mu_{p^{\infty}}))[p^{\infty}]$
is killed by $p$.
Denote by $\alpha$ the non-unit root of 
the equation $T^2-a_p(E)T+p=0$.
Then, $\alpha$ is a uniformizer of $\mbb{Q}_p$
and the $G_{\mbb{Q}_p}$-action on 
the Tate module $T_p(\hat{E})$ of $\hat{E}$  
is given by the Lubin-Tate character 
$\chi\colon G_{\mbb{Q}_p}\to \mbb{Z}_p^{\times}$ associated with $\alpha$
by Lemma \ref{lem:red:def}.
Now take any $P\in \hat{E}(\mbb{Q}_p(\mu_{p^{\infty}}))[p^{\infty}]$.
Then $(\chi(\sigma)-1)P=0$ for every $\sigma\in G_{\mbb{Q}_p(\mu_{p^{\infty}})}$.
By abuse of notation, we also denote by $\chi$
the composite of $\chi$ (considered as a character of $G_{\mbb{Q}_p}^{\mrm{ab}}$)
and the local reciprocity map $\mbb{Q}_p^{\times}\to G_{\mbb{Q}_p}^{\mrm{ab}}$.
Here, $G_{\mbb{Q}_p}^{\mrm{ab}}$ is the maximal abelian quotient of 
$G_{\mbb{Q}_p}$.
If we denote by $v_p$ the $p$-adic valuation normalized by $v_p(p)=1$, 
then we have 
\begin{equation}
\mrm{Min}\{v_p(\chi(\sigma)-1)\mid \sigma\in G_{\mbb{Q}_p(\mu_{p^{\infty}})}\}
\le v_p(\chi(p^{-1})-1)
= v_p(p\alpha^{-1}-1)
\end{equation}
by \cite[Proposition 2.1]{Oze24-2}.
Note that $\beta:=p\alpha^{-1}$ 
is the unit root of the equation $T^2-a_p(E)T+p=0$.
Since $0=\beta^2-a_p(E)\beta+p=
(\beta-1)(\beta-a_p(E)+1)+\#\bar{E}(\mbb{F}_p)$,
we have 
\begin{equation}
v_p(p\alpha^{-1}-1)\le v_p(\#\bar{E}(\mbb{F}_p))\le 1
\end{equation}
by the Hasse bound.
Therefore, we obtain $pP=0$
for every $P\in \hat{E}(\mbb{Q}_p(\mu_{p^{\infty}}))[p^{\infty}]$
as desired.
\end{proof}

\begin{proposition}
\label{Main:cycl}
Assume $p\ge 3$. Let $E$ be an elliptic curve over $\mbb{Q}_p$ with good ordinary reduction. Then, we have 

\begin{align*}
E(\mbb{Q}_p(\mu_{p^{\infty}}))[p^{\infty}]
\simeq 
\left\{ \,
\begin{aligned}
& 0 & if & \ \mbox{$\bar{E}(\mbb{F}_p)[p]=0$},\\
& \mbb{Z}/p\mbb{Z} & if & \ \mbox{$\bar{E}(\mbb{F}_p)[p]\not=0$ and 
$E(\mbb{Q}_p)[p]=0$},\\
& \mbb{Z}/p\mbb{Z}\times \mbb{Z}/p\mbb{Z} & if & \ \mbox{$\bar{E}(\mbb{F}_p)[p]\not=0$ and $E(\mbb{Q}_p)[p]\not=0$}.
\end{aligned}
\right.
\end{align*}
Furthermore, if $\bar{E}(\mbb{F}_p)[p]\not=0$,
then the minimum fields of definition of 
$E(\mbb{Q}_p(\mu_{p^{\infty}}))[p^{\infty}]$ and 
$E(\mbb{Q}_p(\mu_{p^{\infty}}))_{\mrm{tor}}$ are  
$\mbb{Q}_p(\mu_p)$.
\end{proposition}

\begin{proof}
Let us first consider the case where $\bar{E}(\mbb{F}_p)[p]=0$.
Let $\chi\colon G_{\mbb{Q}_p}\to \mbb{F}_p^{\times}$ (resp.\ $\psi\colon G_{\mbb{Q}_p}\to \mbb{F}_p^{\times}$) be the characters 
obtained by the $G_{\mbb{Q}_p}$-action on $\hat{E}[p]$
 (resp.\ $\bar{E}[p]$).
 Since $\bar{E}(\mbb{F}_p)[p]=0$,
 $\psi$ is not trivial.
 Since $\psi$ is unramified and $\chi \psi$ coincides with mod $p$
 cyclotomic character,
 we see that $\chi$ is not trivial 
on $G_{\mbb{Q}_p(\mu_{p^{\infty}})}$.
Thus we have $\hat{E}(\mbb{Q}_p(\mu_{p^{\infty}}))[p]=0$.
Now the result immediately follows
from  an exact sequence $0\to \hat{E}[p]\to E[p]\to \bar{E}[p]\to 0$ of $\mbb{F}_p[G_{\mbb{Q}_p}]$-modules.

Next we consider the case where 
$\bar{E}(\mbb{F}_p)[p]\not=0$ and 
$E(\mbb{Q}_p)[p]=0$.
Since $G_{\mbb{Q}_p}$ acts on  $\bar{E}[p]$ trivial, 
we have isomorphisms 
$\hat{E}[p]\simeq \mbb{F}_p(1)$
and $\bar{E}[p]\simeq \mbb{F}_p$
of $G_{\mbb{Q}_p}$-modules.
Thus there exists a natural exact sequence 
\begin{equation}
\label{ext1}
0\to \mbb{F}_p(1)\to E[p]\to \mbb{F}_p\to 0
\end{equation} 
of $\mbb{F}_p[G_{\mbb{Q}_p}]$-modules.
In particular, $E(\mbb{Q}_p(\mu_p))[p]$
contains a submodule $\hat{E}(\mbb{Q}_p(\mu_p))[p]$
of order $p$.
Now we assume that $E(\mbb{Q}_p(\mu_{p^n}))[p]$
is isomorphic to $\mbb{Z}/p\mbb{Z}\times \mbb{Z}/p\mbb{Z}$
for some $n>0$.
Then the extension \eqref{ext1} splits
as $\mbb{F}_p[G_{\mbb{Q}_p(\mu_{p^n})}]$-modules.
Since the kernel of the restriction map 
$H^1(\mbb{Q}_p,\mbb{F}_p(1))\to 
H^1(\mbb{Q}_p(\mu_{p^{n}}),\mbb{F}_p(1))$,
which is isomorphic to $H^1(\mbb{Q}_p(\mu_{p^{n}})/\mbb{Q}_p,\mbb{F}_p(1))$, 
is trivial, 
the extension \eqref{ext1} splits
as $\mbb{F}_p[G_{\mbb{Q}_p}]$-modules.
Thus we have $E(\mbb{Q}_p)[p]=\mbb{Z}/p\mbb{Z}$
but this contradicts the assumption that 
$E(\mbb{Q}_p)[p]=0$.
Hence, we obtain  $E(\mbb{Q}_p(\mu_{p^n}))[p]\simeq \mbb{Z}/p\mbb{Z}$ for any $n>0$.
This shows 
$$
E(\mbb{Q}_p(\mu_{p^{\infty}}))[p]
=E(\mbb{Q}_p(\mu_{p}))[p]
=\hat{E}(\mbb{Q}_p(\mu_p))[p]
$$
and these are isomorphic to $\mbb{F}_p(1)$ 
as $G_{\mbb{Q}_p}$-modules.
It follows from Lemma \ref{p-bound} that 
$E(\mbb{Q}_p(\mu_{p^{\infty}}))[p^{\infty}]$
is isomorphic to either $\mbb{Z}/p\mbb{Z}$ or $\mbb{Z}/p^2\mbb{Z}$
as modules.
Therefore, for the proof, it suffices to show that 
$E(\mbb{Q}_p(\mu_{p^{\infty}}))[p^{\infty}]$
is not isomorphic to $\mbb{Z}/p^2\mbb{Z}$.
Assume that $E(\mbb{Q}_p(\mu_{p^{\infty}}))[p^{\infty}]\simeq \mbb{Z}/p^2\mbb{Z}$.
Consider the following commutative diagram.
\[
\xymatrix{
0 \ar[r] 
& \hat{E}(\mbb{Q}_p(\mu_{p^{\infty}}))[p^{\infty}] \ar[r]
& E(\mbb{Q}_p(\mu_{p^{\infty}}))[p^{\infty}]
 \ar[r] 
& \bar{E}(\mbb{F}_p)[p^{\infty}] \\
0 \ar[r] 
& \hat{E}(\mbb{Q}_p(\mu_{p}))[p]
\ar[r] \ar@{^{(}-_>}[u] 
& E(\mbb{Q}_p(\mu_{p}))[p]
\ar[r]  \ar@{^{(}-_>}[u] 
& \bar{E}(\mbb{F}_p)[p]. 
 \ar@{^{(}-_>}[u] }
\]
By Lemma \ref{p-bound}, 
the left vertical arrow is bijective. 
By the Hasse bound, 
the right vertical arrow is bijective. 
Thus we find that the reduction map 
$E(\mbb{Q}_p(\mu_{p^{\infty}}))[p^{\infty}]
\to \bar{E}(\mbb{F}_p)[p^{\infty}]$ is surjective.
Applying Lemma \ref{Z/p2Z} below with 
$G=G_{\mbb{Q}_p}$ and  $M=E(\mbb{Q}_p(\mu_{p^{\infty}}))[p^{\infty}]$,
we obtain an isomorphism 
$\hat{E}(\mbb{Q}_p(\mu_{p^{\infty}}))[p^{\infty}]
\simeq \bar{E}(\mbb{F}_p)[p^{\infty}]$
of $G_{\mbb{Q}_p}$-modules.
This gives  
$$
\mbb{F}_p(1)\simeq \hat{E}[p]
=\hat{E}(\mbb{Q}_p(\mu_{p^{\infty}}))[p^{\infty}]
\simeq \bar{E}(\mbb{F}_p)[p^{\infty}]
=\bar{E}(\mbb{F}_p)[p]\simeq \mbb{F}_p
$$
as $G_{\mbb{Q}_p}$-modules but this is a contradiction.
Therefore, we obtain $E(\mbb{Q}_p(\mu_{p^{\infty}}))[p^{\infty}]\simeq \mbb{Z}/p\mbb{Z}$
as desired.
Moreover, since we also showed that 
 $E(\mbb{Q}_p(\mu_{p^{\infty}}))[p^{\infty}] 
=\hat{E}(\mbb{Q}_p(\mu_{p}))[p]
\simeq  \mbb{F}_p(1)$ as $G_{\mbb{Q}_p}$-modules,
we find that the definition field of $E(\mbb{Q}_p(\mu_{p^{\infty}}))[p^{\infty}]$
is $\mbb{Q}_p(\mu_p)$.
Note that the fields of definition of $E(\mbb{Q}_p(\mu_{p^{\infty}}))[p^{\infty}]$ 
coincides with that of 
$E(\mbb{Q}_p(\mu_{p^{\infty}}))_{\mrm{tor}}$
since the prime-to-$p$ part of $E(\mbb{Q}_p(\mu_{p^{\infty}}))_{\mrm{tor}}$
is rational over $\mbb{Q}_p$ by the N\'eron-Ogg-Shafarevich criterion.

Finally we consider the case where 
$\bar{E}(\mbb{F}_p)[p]\not=0$ and 
$E(\mbb{Q}_p)[p]\not=0$.
Since $\hat{E}[p]\, (\simeq \mbb{F}_p(1))$ and 
$E(\mbb{Q}_p)[p]\, (\simeq \mbb{F}_p)$
are non-isomorphic $G_{\mbb{Q}_p}$-submodules
of $E[p]$, 
we have isomorphisms 
$$
E[p]=\hat{E}[p]\oplus E(\mbb{Q}_p)[p]
\simeq \mbb{F}_p(1)\oplus \mbb{F}_p
$$
of $G_{\mbb{Q}_p}$-submodules. 
In particular,
we have $E(\mbb{Q}_p(\mu_{p^{\infty}}))[p]=E[p]$.
Since the order of $E(\mbb{Q}_p(\mu_{p^{\infty}}))[p^{\infty}]$
is at most $p^2$ by Lemma \ref{p-bound},
we obtain $E(\mbb{Q}_p(\mu_{p^{\infty}}))[p^{\infty}]=E[p]$.
In particular, the definition field of $E(\mbb{Q}_p(\mu_{p^{\infty}}))[p^{\infty}]$
is $\mbb{Q}_p(\mu_p)$.
As we have seen above, it follows from 
the N\'eron-Ogg-Shafarevich criterion that the fields of definition of 
$E(\mbb{Q}_p(\mu_{p^{\infty}}))_{\mrm{tor}}$ is also $\mbb{Q}_p(\mu_p)$.
\end{proof}

In the proof above, we used the following lemma.
\begin{lemma}
\label{Z/p2Z}
Let $G$ be a group, $p$ a prime (including the case $p=2$) and $n>0$ an integer.
Let $M$ be a $\mbb{Z}/p^{2n}\mbb{Z}[G]$-module 
which is free of finite rank over $\mbb{Z}/p^{2n}\mbb{Z}$.
Then, we have a canonical isomorphism
$p^nM\simeq M/p^nM$
of $\mbb{Z}/p^n\mbb{Z}[G]$-modules.
\end{lemma}

\begin{proof}
The result immediately follows by applying the snake lemma to
the commutative diagram of $G$-modules below:
\[
\xymatrix{
0 \ar[r] 
& p^nM \ar[r] \ar^{p^n}[d]
& M \ar[r] \ar^{p^n}[d]
& M/p^nM \ar[r] \ar^{p^n}[d]
& 0 \\
0 \ar[r] 
& p^nM \ar[r] 
& M \ar[r]  
& M/p^nM \ar[r]
& 0. 
}
\]
\end{proof}

\begin{proof}[Proof of the first part of Theorem \ref{E(Qp(mu)):good} (2)]
The N\'eron-Ogg-Shafarevich criterion shows that 
the prime-to-$p$ parts of 
$E(\mbb{Q}_p(\mu_{p^{\infty}}))_{\mrm{tor}}$
and $E(\mbb{Q}_p(\mu_{p}))_{\mrm{tor}}$
coincide with each other.
Moreover, we have  
$E(\mbb{Q}_p(\mu_{p^{\infty}}))[p^{\infty}]
=E(\mbb{Q}_p(\mu_{p}))[p^{\infty}]$
by Proposition \ref{Main:cycl}. 
This finishes a proof.
\end{proof}

Let us prove the second part of Theorem \ref{E(Qp(mu)):good} (2).
The statement is equivalent to say that, for an elliptic curve $E$  over $\mbb{Q}_p$ 
with good ordinary reduction, then $E(\mbb{Q}_p(\mu_{p^{\infty}}))_{\mrm{tor}}$
is isomorphic to one of the following groups.
\begin{enumerate}
\item[{\rm (a)}] $G_{m,k}$,   $(m,k)\in I_{\mrm{ord}}$,
\item[{\rm (b)}] $G_{p,1}$,
\item[{\rm (c)}] $G_{p,2}$ with $p\le 5$,
\end{enumerate}
and each of these groups appears as $E(\mbb{Q}_p(\mu_{p^{\infty}}))_{\mrm{tor}}$ for 
some elliptic curve $E$ over $\mbb{Q}_p$ with good ordinary reduction.

\begin{lemma}
\label{cycl:lem}
Let $E$ be an elliptic curve  over $\mbb{Q}_p$ 
with good ordinary reduction.
\begin{itemize}
\item[{\rm (1)}] If $E(\mbb{Q}_p)[p]=0$,
then $E(\mbb{Q}_p(\mu_{p^{\infty}}))_{\mrm{tor}}\simeq \bar{E}(\mbb{F}_p)$
as abstract groups.
\item[{\rm (2)}] If $E(\mbb{Q}_p)[p]\not=0$,
then $E(\mbb{Q}_p(\mu_{p^{\infty}}))_{\mrm{tor}}\simeq \bar{E}(\mbb{F}_p)\times \mbb{Z}/p\mbb{Z}$
as abstract groups.
\end{itemize}

\end{lemma}

\begin{proof}
Assume that an elliptic curve $E$ over $\mbb{Q}_p$ 
has good ordinary reduction.
If $E(\mbb{Q}_p)[p]=0$ (resp.\ $E(\mbb{Q}_p)[p]\not=0$),
we have $E(\mbb{Q}_p(\mu_{p^{\infty}}))[p^{\infty}]
\simeq \bar{E}(\mbb{F}_p)[p^{\infty}]$
(resp.\ $E(\mbb{Q}_p(\mu_{p^{\infty}}))[p^{\infty}]
\simeq \bar{E}(\mbb{F}_p)[p^{\infty}]\times \mbb{Z}/p\mbb{Z}$)
by Proposition \ref{Main:cycl}.
Since the reduction map gives an isomorphism between the prime-to-$p$
parts of $E(\mbb{Q}_p(\mu_{p^{\infty}}))$ and $\bar{E}(\mbb{F}_p)$,
the result follows.
\end{proof}

\begin{proof}[Proof of the second part of Theorem \ref{E(Qp(mu)):good} (2)]
Suppose that an elliptic curve $E$ over $\mbb{Q}_p$ 
has good ordinary reduction. 
If $E(\mbb{Q}_p)[p]=0$,
then it follows from Lemma \ref{cycl:lem} (1) that  $E(\mbb{Q}_p(\mu_{p^{\infty}}))$
is isomorphic to $G_{m,k}$ for some  $(m,k)\in I_{\mrm{ord}}$.
If $E(\mbb{Q}_p)[p]\not=0$, 
then it follows from Lemma \ref{cycl:lem} (2) that  $E(\mbb{Q}_p(\mu_{p^{\infty}}))_{\mrm{tor}}$
is isomorphic to $G_{m,k}\times \mbb{Z}/p\mbb{Z}$ 
for some  $(m,k)\in I_{\mrm{ord}}$.
Note that we moreover have $p\mid k$ since 
$\bar{E}(\mbb{F}_p)[p]$ is not zero by \eqref{p-part}.
In this case we see  $(m,k)=(1,jp)$ with $j\in \{1,2\}$ (resp.\ $j=1$ )
for $p\le 5$ (resp.\ $p>5$),
 and then $E(\mbb{Q}_p(\mu_{p^{\infty}}))_{\mrm{tor}}$
is isomorphic to 
and $G_{m,k}\times \mbb{Z}/p\mbb{Z}\simeq G_{p,j}$. 
Therefore, we showed that  
$E(\mbb{Q}_p(\mu_{p^{\infty}}))_{\mrm{tor}}$
is isomorphic to one of the groups appearing  in (a), (b) or (c).

Conversely, let $G$ be a group appearing in (a), (b) or (c).
Suppose $G=G_{m,k}$ as in (a) and suppose in addition $p\nmid k$.
By Theorem \ref{E(Fp)}, there exists an ordinary elliptic curve
$\bar{E}$ over $\mbb{F}_p$ such that  $G\simeq \bar{E}(\mbb{F}_p)$.
Take any lift $E$ of $\bar{E}$ to $\mbb{Q}_p$.
By $p\nmid k$,  $\bar{E}(\mbb{F}_p)[p]$ is trivial, and thus  
we have $E(\mbb{Q}_p)[p]=0$ by \eqref{p-part}. 
By Lemma \ref{cycl:lem}, we have 
$E(\mbb{Q}_p(\mu_{p^{\infty}}))_{\mrm{tor}}\simeq G$. 
Next we suppose one of the following situations.
\begin{itemize}
\item $G=G_{m,k}$ as in (a) and suppose in addition $p\mid k$. In this case $G=G_{1,jp}$ for $j\in \{1,2\}$ (resp.\ $j=1$ )
for $p\le 5$ (resp.\ $p>5$).
\item $G=G_{p,j}\ (\simeq G_{1,jp}\times \mbb{Z}/p\mbb{Z})$ is as in (b) or (c).
\end{itemize}
By Theorem \ref{E(Fp)}, there exists an ordinary elliptic curve
$\bar{E}$ over $\mbb{F}_p$ such that  $G_{1,jp}\simeq \bar{E}(\mbb{F}_p)$.
By Lemma 3.1 and Lemma 3.2 of \cite{DaWe08},
there exist elliptic curves $E_1$ and $E_2$ over $\mbb{Q}_p$
whose reductions are $\bar{E}$ 
such that $E_1(\mbb{Q}_p)[p]=0$ and
$E_2(\mbb{Q}_p)[p]\not=0$.
(Note that the canonical lift of $\bar{E}$
satisfies the desired condition for $E_2$.)
It follows from Lemma \ref{cycl:lem} that 
$E_1(\mbb{Q}_p(\mu_{p^{\infty}}))\simeq G_{1,jp}$
and $E_2(\mbb{Q}_p(\mu_{p^{\infty}}))\simeq G_{p,j}$.
This finishes a proof.
\end{proof}

\subsection{The case $p=2$}\label{subsection:p=2}

We show Theorem \ref{E(Qp(mu)):good} (3)
and Theorem \ref{E(Q2(mu4)):good}. 
We begin with a proof of the second statement
of Theorem \ref{E(Qp(mu)):good} (3); 
it suffices to show that, 
for  an elliptic curve $E$ over $\mbb{Q}_2$
with good ordinary reduction, then $E(\mbb{Q}_2(\mu_{2^{\infty}}))_{\mrm{tor}}$
is isomorphic to one of the following groups.
\begin{enumerate}
\item[{\rm (a)}] $G_{1,k}$,   $k\in \{4,8 \}$,
\item[{\rm (b)}] $G_{2,k}$,   $k\in \{1,4\}$, 
\item[{\rm (c)}] $G_{4,1}$,
\end{enumerate}
and each of these groups appears as $E(\mbb{Q}_2(\mu_{2^{\infty}}))_{\mrm{tor}}$ for 
some elliptic curve $E$ over $\mbb{Q}_2$ with good ordinary reduction.

For our proof below,
we need Fontaine's results on 
ramification theory of finite flat commutative group schemes. 
We give a brief sketch here
(with restricting $2$-adic cases); 
see \cite{Ser68} and \cite{Fon85}
for more precise information.
Let $K$ be a $2$-adic field and 
$L/K$ be a (not necessarily finite) Galois extension.
For any non-negative real number $u\ge 0$,
let $\mrm{Gal}(L/K)^{(u)}$ be the $u$-th upper ramification subgroup of $\mrm{Gal}(L/K)$ in the sense of \cite{Fon85}.
For a finite Galois extension $L/K$, we define 
{\it the maximal upper ramification break} of $L/K$ defined by 
$u_{L/K}=\sup \{ u \in \mbb{R}\, |\, \mrm{Gal}(L/K)^{(u)} \not=1\}$. 
It is well-known that 
\begin{itemize}
\item[--] $L/K$ is unramified if and only if $u_{L/K}=0$, 
\item[--] $L/K$ is tamely ramified if and only if $u_{L/K}\le 1$,  
and 
\item[--] $L/K$ is wildly ramified if and only if $u_{L/K}> 1$.
\end{itemize}
For example, there exist 
7 quadratic extensions of $\mbb{Q}_2$,
and their maximal upper ramification breaks are given in the Table 
\ref{table:degree2}.
We set 
$G_K^{(u)}:=\mrm{Gal}(\overline{\mbb{Q}}_2/K)^{(u)}$.
It is shown by Fontaine \cite[Section 2, Th\'eor\`em 1]{Fon85} 
that,
for any finite flat commutative group scheme 
$\mcal{G}$ over $K$
killed by $2^n$,
the group $G_K^{(u)}$ acts trivial on 
$\mcal{G}(\overline{K})$ for $u>e_K(n+1)$.
This is equivalent to say that, 
if we denote by $L/K$ the Galois extension corresponding to 
the kernel of the $G_K$-action on 
$\mcal{G}(\overline{K})$, 
then $\mrm{Gal}(L/K)^{(u)}$ is trivial for $u>e_K(n+1)$,
that is, $u_{L/K}\le e_K(n+1)$.
By applying the result for $E[2^n]$, 
we have $u_{L/\mbb{Q}_2} \leq n+1$ for $L=\mbb{Q}_2(E[2^n])$. 
Since $u_{\mbb{Q}_2(\mu_{2^n})/\mbb{Q}_2}=n$ for any integer $n>1$ 
(cf.\ \cite[Chap.\ IV, Sect.\ 4, Cor.]{Ser68}), 
we have 
\begin{equation}
\label{Fontaine:bound}
E(\mbb{Q}_2(\mu_{2^{\infty}}))[2^n]
=E(\mbb{Q}_2(\mu_{2^{n+1}}))[2^n] 
\end{equation}

Let us return to the proof of
the second statement of 
Theorem \ref{E(Qp(mu)):good} (3). 
We need the following lemma.

\begin{lemma}
\label{hat=bar}
For an elliptic curve $E$ over $\mbb{Q}_2$ with good ordinary 
reduction, we have $\hat{E} 
(\mbb{Q}_2(\mu_{2^{\infty}}))_{\mrm{tor}}\simeq \bar{E}(\mbb{F}_2)$ as 
abstract groups.
\end{lemma}

\begin{proof}
First we note that 
$\bar{E}(\mbb{F}_2)$ contains 
the element of order $2$ since 
the Galois group $G_{\mbb{Q}_2}$ 
acts on $\bar{E}[2]\, (\simeq \mbb{Z}/2\mbb{Z})$ trivial.
The Hasse bound shows that 
$\bar{E}(\mbb{F}_2)$ is isomorphic to either $\mbb{Z}/2\mbb{Z}$ or $\mbb{Z}/4\mbb{Z}$.
Thus 
\begin{align*}
a_2(E)=
\left\{ \,
\begin{aligned}
& 1 & if & \ \bar{E}(\mbb{F}_2)\simeq \mbb{Z}/2\mbb{Z},\\
& -1 & if & \ \bar{E}(\mbb{F}_2)\simeq\mbb{Z}/4\mbb{Z}.
\end{aligned}
\right.
\end{align*}
Denote by $\alpha$ 
(resp.\ $\beta$)
the non-unit root (resp.\ unit root) 
of the equation $T^2-a_2(E)T+2=0$,
where $a_2(E)=1+2-\# \bar{E}(\mbb{F}_2)$. 
Then $G_{\mbb{Q}_2}$ acts on the Tate module $V_2(\hat{E})$ of $\hat{E}$ by 
the Lubin-Tate character $\chi$ 
associated with the uniformizer $\alpha$
by Lemma \ref{lem:red:def}.
By abuse of notation  
we also denote by 
$\chi\colon \mbb{Q}_2^{\times}\to \mbb{Q}_2^{\times}$
the composite of 
$\chi$ (considered as a character of $G^{\mrm{ab}}_{\mbb{Q}_2}$) 
and the local reciprocity map $\mbb{Q}_2^{\times }\to G^{\mrm{ab}}_{\mbb{Q}_2}$. Here, 
$G^{\mrm{ab}}_{\mbb{Q}_2}=\mrm{Gal}(\mbb{Q}_2^{\mrm{ab}}/\mbb{Q}_2(\mu_{2^{\infty}}))$ is the maximal abelian quotient of $G_{\mbb{Q}_2}$.
Then 
$\chi(2)=\chi(\alpha \beta) 
= \beta^{-1}\equiv -a_2(E)$ mod $4$.
Since the subgroup of $\mbb{Q}_2^{\times}$ corresponding to 
$\mrm{Gal}(\mbb{Q}_2^{\mrm{ab}}/\mbb{Q}_2(\mu_{2^{\infty}}))$ via the local reciprocity map
is the closure of the group generated by $2$,
we obtain that 
\begin{align*}
\left\{ \,
\begin{aligned}
& \mbox{$\chi\not\equiv 1$ mod $4$} & if & \ \bar{E}(\mbb{F}_2)\simeq \mbb{Z}/2\mbb{Z},\\
& \mbox{$\chi\equiv 1$ mod $4$}  & if & \ \bar{E}(\mbb{F}_2)\simeq\mbb{Z}/4\mbb{Z}
\end{aligned}
\right.
\end{align*}
on $G_{\mbb{Q}_2(\mu_{2^{\infty}})}$.
Therefore, we see  
$\hat{E}(\mbb{Q}_p(\mu_{2^{\infty}}))[2^{\infty}]\simeq \bar{E}(\mbb{F}_2)
$
as abstract groups.
Since orders of torsion elements of $\hat{E}$ are power of $2$, 
we finish the proof of the lemma.
\end{proof}

\begin{proof}[Proof of the second statement of 
Theorem \ref{E(Qp(mu)):good} (3)]
Consider an exact sequence 
$$
0\to \hat{E}(\mbb{Q}_p(\mu_{2^{\infty}}))_{\mrm{tor}}\to 
E(\mbb{Q}_p(\mu_{2^{\infty}}))_{\mrm{tor}}
\to \bar{E}(\mbb{F}_2)
$$
of $G_{\mbb{Q}_2}$-modules.
It follows from Lemma \ref{hat=bar} that 
$\hat{E}(\mbb{Q}_p(\mu_{2^{\infty}}))_{\mrm{tor}}\simeq \bar{E}(\mbb{F}_2)$
as abstract groups,
and the orders of these groups 
are $2$ or $4$ by the Hasse bound. 
This shows that  $E(\mbb{Q}_2(\mu_{2^{\infty}}))_{\mrm{tor}}$ is isomorphic 
to one of the groups appearing in the following.
\begin{enumerate}
\item[{\rm (a)'}] $G_{1,k}$,   $k\in \{2,4,8,16 \}$,
\item[{\rm (b)'}] $G_{2,k}$,   $k\in \{1,2,4\}$,
\item[{\rm (c)}] $G_{4,1}$.
\end{enumerate}
We claim that $E(\mbb{Q}_2(\mu_{2^{\infty}}))_{\mrm{tor}}$ is 
not isomorphic to $G_{1,16}$. 
Assume $E(\mbb{Q}_2(\mu_{2^{\infty}}))_{\mrm{tor}}
\simeq \mbb{Z}/16\mbb{Z}$.
If this is the case, putting $M=E(\mbb{Q}_2(\mu_{2^{\infty}}))_{\mrm{tor}}$,
we have $4M=\hat{E}(\mbb{Q}_2(\mu_{2^{\infty}}))[2^{\infty}]$ and 
$M/4M=\bar{E}(\mbb{F}_2)$.
It follows from Lemma \ref{Z/p2Z} that we have 
an isomorphism $\hat{E}(\mbb{Q}_2(\mu_{2^{\infty}}))[2^{\infty}]\simeq \bar{E}(\mbb{F}_2)$
of $G_{\mbb{Q}_2}$-modules but this is a contradiction since 
$G_{\mbb{Q}_2}$ acts on $\hat{E}(\mbb{Q}_2(\mu_{2^{\infty}}))[2^{\infty}]=\hat{E}(\mbb{Q}_2(\mu_{2^{\infty}}))[4]\, 
(\simeq \mbb{Z}/4\mbb{Z})$ by the $2$-adic cyclotomic character modulo $4$.

By the claim above, 
$E(\mbb{Q}_2(\mu_{2^{\infty}}))_{\mrm{tor}}$ is killed by $2^3$.
By \eqref{Fontaine:bound}, we see that 
$E(\mbb{Q}_2(\mu_{2^{\infty}}))_{\mrm{tor}}
=E(\mbb{Q}_2(\mu_{16}))_{\mrm{tor}}$.
Since we have a descent from $\mbb{Q}_2(\mu_{2^{\infty}})$ 
to a (not so large) finite extension 
$\mbb{Q}_2(\mu_{16})$ of $\mbb{Q}_2$,
we can apply a computational approach; 
by MAGMA calculation with Algorithm \ref{alg:torsion}, 
we can check that 
some of the groups in (a)', (b)' or (c) above appears 
as $E(\mbb{Q}_p(\mu_{p^{\infty}}))_{\mrm{tor}}$
for some elliptic curve $E$ over $\mbb{Q}_2$ with good ordinary reduction:
\begin{itemize}
\item $E=$\href{https://www.lmfdb.org/EllipticCurve/Q/33/a/1}{33.a3} satisfies $E(\mbb{Q}_2(\mu_{2^{\infty}}))_{\mrm{tor}}\simeq G_{1,4}$.
\item $E=$\href{https://www.lmfdb.org/EllipticCurve/Q/15/a/1}{15.a5} satisfies $E(\mbb{Q}_2(\mu_{2^{\infty}}))_{\mrm{tor}}\simeq G_{1,8}$.
\item $E=$\href{https://www.lmfdb.org/EllipticCurve/Q/15/a/2}{33.a1} satisfies $E(\mbb{Q}_2(\mu_{2^{\infty}}))_{\mrm{tor}}\simeq G_{2,1}$.
\item $E=$\href{https://www.lmfdb.org/EllipticCurve/Q/15/a/2}{15.a2} satisfies $E(\mbb{Q}_2(\mu_{2^{\infty}}))_{\mrm{tor}}\simeq G_{2,4}$.
\item $E=$\href{https://www.lmfdb.org/EllipticCurve/Q/15/a/5}{15.a1} satisfies $E(\mbb{Q}_2(\mu_{2^{\infty}}))_{\mrm{tor}}\simeq G_{4,1}$.
\end{itemize}

For the proof of the theorem, it suffices to show that 
there is no elliptic curve $E$ over $\mbb{Q}_2$ with good ordinary reduction
such that  $E(\mbb{Q}_p(\mu_{p^{\infty}}))_{\mrm{tor}}$ is isomorphic 
to either $G_{1,2}$ or $G_{2,2}$.
In the rest of the proof, we denote by 
$\chi\colon G_{\mbb{Q}_2}\to \mbb{Z}_2^{\times}$ 
the crystalline character defined by the $G_{\mbb{Q}_p}$-action on 
the $p$-adic Tate module $T_p(\hat{E})$ of 
the formal group associated with $E$.
We also denote by $\psi\colon G_{\mbb{Q}_2}\to \mbb{Z}_2^{\times}$ 
the unramified character defined by the $G_{\mbb{Q}_p}$-action on 
the $p$-adic Tate module $T_p(\bar{E})$ of 
the reduction $\bar{E}$ of $E$.
The Weil pairing shows that $\chi \psi=\chi_2$ where $\chi_2$ 
is the $2$-adic cyclotomic character.
Thus we have $\chi \psi \ \mrm{mod}\ 2^n = 1 $ on $G_{\mbb{Q}_2(\mu_{2^n})}$
for each $n>0$.

\vspace{5mm}
{(I) Non-existence of $G_{1,2}$:} 
If $E(\mbb{Q}_2(\mu_{2^{\infty}}))_{\mrm{tor}}
\simeq \mbb{Z}/2\mbb{Z}$
for some elliptic curve $E$ over $\mbb{Q}_2$ with good reduction,
it follows from 
Fontaine's ramification bound 
\eqref{Fontaine:bound} that 
$
E(\mbb{Q}_2(\mu_{2^{\infty}}))_{\mrm{tor}}
=E(\mbb{Q}_2(\mu_4))_{\mrm{tor}}.
$
Hence, it suffices to show that 
$E(\mbb{Q}_2(\mu_{4}))_{\mrm{tor}}
\not\simeq \mbb{Z}/2\mbb{Z}$
for any elliptic curve $E$ over $\mbb{Q}_2$ with good ordinary reduction.

Assume that 
$E(\mbb{Q}_2(\mu_{4}))_{\mrm{tor}} \simeq \mbb{Z}/2\mbb{Z}$
for some elliptic curve $E$ over $\mbb{Q}_2$ with 
good ordinary reduction.
For a suitable choice of a $\mbb{Z}/4\mbb{Z}$-basis of $E[4]$,
the $G_{\mbb{Q}_2}$-action on $E[4]$ is given by 
$$
\rho_{E[4]}=
\begin{pmatrix}
\chi\ \mrm{mod}\ 4 & u \\
0 & \psi\ \mrm{mod}\ 4
\end{pmatrix}
\colon G_{\mbb{Q}_2}\to GL_2(\mbb{Z}/4\mbb{Z})
$$
for some map $u\colon G_{\mbb{Q}_2}\to \mbb{Z}/4\mbb{Z}$.

We claim that 
\begin{equation*}
[\mbb{Q}_2(E[4]):\mbb{Q}_2]=16.
\end{equation*}
Since $\chi \equiv \psi \ \mrm{mod}\ 4$ on $G_{\mbb{Q}_2(\mu_4)}$,
we may regard $H:=\mrm{Gal}(\mbb{Q}_2(E[4])/\mbb{Q}_2(\mu_4))$ 
as a subgroup of 
$$
G:=\left\{ 
\begin{pmatrix}
a & b \\
0 & a
\end{pmatrix}
\mid a\in (\mbb{Z}/4\mbb{Z})^{\times}, b\in \mbb{Z}/4\mbb{Z} 
\right\}
$$
via $\rho_{E[4]}$.
Since $E(\mbb{Q}_2(\mu_4))\not\supset E[2]$, we have 
$u \ \mrm{mod}\ 2 \not=0$ on $G_{\mbb{Q}_2(\mu_4)}$.
Thus $H$ contains  at least either 
$\begin{pmatrix}
1 & 1 \\
0 & 1
\end{pmatrix}$
or 
$\begin{pmatrix}
-1 & 1 \\
0 & -1
\end{pmatrix}$.
If we assume that  $H$ is generated by 
one of these matrices, 
we find that  $E(\mbb{Q}_2(\mu_4)))$ must contain an element of order $4$
but this is a contradiction.
Thus we have $H=G$.
Now the claim immediately follows.

By considering from the view point of ramification, 
we show below that 
$$
[\mbb{Q}_2(E[4]):\mbb{Q}_2]<16
$$ 
holds
(of course this is a contradiction).
Since $\psi$ is unramified, 
$\psi \ \mrm{mod}\ 4$ is trivial on $G_F$
where $F$ is the unramified quadratic extension field $F$ of $\mbb{Q}_2$.
In particular, we have $\chi\ \mrm{mod}\ 4 = \chi_2\ \mrm{mod}\ 4$ on $G_F$. 
Since $E(\mbb{Q}_2(\mu_{4}))$ does not contain $E[2]$,
we have $u\ \mrm{mod}\ 2\not=0$ on $G_{\mbb{Q}_2(\mu_{4})}$.
Thus the field $L$ corresponding to 
the kernel of $u\ \mrm{mod}\ 2\colon G_{\mbb{Q}_2}\to \mbb{Z}/2\mbb{Z}$
is a quadratic extension of $\mbb{Q}_2$.
Note that we have $L=\mbb{Q}_2(E[2])$ and  
\begin{equation}
\label{umod2}
u(G_L)\subset 2\cdot \mbb{Z}/4\mbb{Z}.
\end{equation}
By \cite[Section 2, Th\'eor\`em 1]{Fon85}, 
the maximal ramification break $u_{L/\mbb{Q}_2}$ of $L/\mbb{Q}_2$
is at most $2$. Hence there are three possibilities for $L$;
$L$ is isomorphic to either 
$L_1=\mbb{Q}_2[x]/(x^2+2x+2)\ (\simeq \mbb{Q}_2(\mu_4))$, 
$L_2=\mbb{Q}_2[x]/(x^2+2x+6)$ or 
$F$ (see Table \ref{table:degree2}).
Since $E(\mbb{Q}_2(\mu_{4}))$
does not contain $E[2]$, 
$L$ is isomorphic to either 
$L_2$ or 
$F$.
\begin{itemize}
\item[(i)] Suppose $L=F$. We have 
$$
\rho_{E[4]}=
\begin{pmatrix}
\chi\ \mrm{mod}\ 4 & u \\
0 & 1
\end{pmatrix}
$$
on $G_F$.
Since the order of $u(G_F)$ is at most $2$ by \eqref{umod2},
we see that $[\mbb{Q}_2(E[4]):F]\le 4$,
which shows $[\mbb{Q}_2(E[4]):\mbb{Q}_2]\le 8<16$ as desired.
\item[(ii)] Suppose $L=L_2$. 
In this case, the field $L(\mu_4)$ cotains $F$  
by Proposition \ref{composite} (1).
Thus we have 
$$
\rho_{E[4]}=
\begin{pmatrix}
1 & u \\
0 & 1
\end{pmatrix}
$$
on $G_{L(\mu_4)}$.
Since the order of $u(G_{L(\mu_4)})$ is at most $2$ by \eqref{umod2},
we see that $[\mbb{Q}_2(E[4]):L(\mu_4)]\le 2$,
which shows $[\mbb{Q}_2(E[4]):\mbb{Q}_2]\le 8<16$ as desired.
\end{itemize}
Therefore,
we finish the proof of (I).

\vspace{5mm}
{(II) Non-existence of $G_{2,2}$:} 
If $E(\mbb{Q}_2(\mu_{2^{\infty}}))_{\mrm{tor}}
\simeq  \mbb{Z}/2\mbb{Z}\times \mbb{Z}/4\mbb{Z}$
for some elliptic curve $E$ over $\mbb{Q}_2$ with good reduction,
it follows from 
Fontaine's ramification bound 
\eqref{Fontaine:bound} that 
$
E(\mbb{Q}_2(\mu_{2^{\infty}}))_{\mrm{tor}}
=E(\mbb{Q}_2(\mu_8))_{\mrm{tor}}.
$
Hence, it suffices to show that 
$E(\mbb{Q}_2(\mu_{8}))_{\mrm{tor}}
\not\simeq  \mbb{Z}/2\mbb{Z}\times \mbb{Z}/4\mbb{Z}$
for any elliptic curve $E$ over $\mbb{Q}_2$ with good ordinary reduction.

Assume that 
$E(\mbb{Q}_2(\mu_{8}))_{\mrm{tor}} \simeq  \mbb{Z}/2\mbb{Z}\times \mbb{Z}/4\mbb{Z}$
for some elliptic curve $E$ over $\mbb{Q}_2$ with 
good ordinary reduction.
For a suitable choice of a $\mbb{Z}/8\mbb{Z}$-basis of $E[8]$,
the $G_{\mbb{Q}_2}$-action on $E[8]$ is given by 
$$
\rho_{E[8]}=
\begin{pmatrix}
\chi\ \mrm{mod}\ 8 & u \\
0 & \psi\ \mrm{mod}\ 8
\end{pmatrix}
\colon G_{\mbb{Q}_2}\to GL_2(\mbb{Z}/8\mbb{Z})
$$
for some map $u\colon G_{\mbb{Q}_2}\to \mbb{Z}/8\mbb{Z}$.
Here we give some remarks on 
the character $\psi\ \mrm{mod}\ 8$ and the map $u$.
By Lemma \ref{hat=bar} and the assumption that $E(\mbb{Q}_2(\mu_{2^{\infty}}))$
is of order $\ge 8$, we have 
$\bar{E}(\mbb{F}_2)\simeq \mbb{Z}/4\mbb{Z}$. 
Thus  any element of $\bar{E}[4]$ is $\mbb{F}_2$-rational 
but 
some element of $\bar{E}[8]$ is not $\mbb{F}_2$-rational.
This gives 
\begin{equation}
\label{psi:relation}
\mbox{$\psi \ \mrm{mod}\ 4 =1$ on $G_{\mbb{Q}_2}$\quad   
and\quad  $\psi\ \mrm{mod}\ 8 \not=1$ on $G_{\mbb{Q}_2(\mu_8)}$.}
\end{equation}
In particular, $\psi\ \mrm{mod}\ 8\colon G_{\mbb{Q}_2}\to (\mbb{Z}/8\mbb{Z})^{\times}$
has values in $\{1,5\}$
and hence $\psi\ \mrm{mod}\ 8\colon G_{\mbb{Q}_2}\to \{1,5\}\, (\subset (\mbb{Z}/8\mbb{Z})^{\times})$
is the  surjective unramified character.
By 
$
E[2]\subset E(\mbb{Q}_2(\mu_8)))_{\mrm{tor}}, 
$
we see that  $u\ \mrm{mod}\ 2$ is trivial on $G_{\mbb{Q}_2(\mu_8)}$,
that is, $u(G_{\mbb{Q}_2(\mu_8)})\subset 2\cdot \mbb{Z}/8\mbb{Z}$.
Note that 
$\rho_{E[4]}=
\begin{pmatrix}
\chi_2 \ \mrm{mod}\ 4 & u\ \mrm{mod}\ 4\  \\
0 & 1
\end{pmatrix}$
on $G_{\mbb{Q}_2}$ by \eqref{psi:relation}.
Since $E(\mbb{Q}_2(\mu_{2^{\infty}}))$ does not contain $E[4]$, 
it holds 
\begin{equation}
\label{u:relation1}
\mbox{$u\ \mrm{mod}\ 4 \not=0$ on $G_{\mbb{Q}_2(\mu_8)}$.}
\end{equation}
On the other hand,
since $u\ \mrm{mod}\ 4$ on $G_{\mbb{Q}_2(E[2])}$ has values in 
$2\cdot \mbb{Z}/4\mbb{Z}$, 
it holds $\chi(\sigma)u(\sigma)\equiv u(\sigma)\ \mrm{mod}\ 4$
for any $\sigma \in G_{\mbb{Q}_2(E[2])}$.
This gives the fact that  $u\ \mrm{mod}\ 4$ on $G_{\mbb{Q}_2(E[2])}$ 
is a homomorphism with 
values in $2\cdot \mbb{Z}/4\mbb{Z}$.

We claim that 
$$
[\mbb{Q}_2(E[8]):\mbb{Q}_2]=32.
$$
\noindent
Since $\chi \equiv \psi \ \mrm{mod}\ 8$ on $G_{\mbb{Q}_2(\mu_8)}$, 
it follows from \eqref{psi:relation} and $u(G_{\mbb{Q}_2(\mu_8)})\subset 2\cdot \mbb{Z}/8\mbb{Z}$ 
that 
we may regard $H:=\mrm{Gal}(\mbb{Q}_2(E[8])/\mbb{Q}_2(\mu_8))$ 
as a subgroup of 
$$
G:=\left\{ 
\begin{pmatrix}
a & b \\
0 & a
\end{pmatrix}
\mid a\in \{1,5\}\subset (\mbb{Z}/8\mbb{Z})^{\times}, b\in 2 \cdot \mbb{Z}/8\mbb{Z} 
\right\}
$$
via $\rho_{E[8]}$.
By \eqref{u:relation1}, $H$ contains at least either 
$\begin{pmatrix}
1 & 2 \\
0 & 1
\end{pmatrix}$
or 
$\begin{pmatrix}
5 & 2 \\
0 & 5
\end{pmatrix}$.
If $H$ is generated by 
one of these matrices, 
we find that  $E(\mbb{Q}_2(\mu_8))$ must contain an element of order $8$
but this is a contradiction.
Thus we have $H=G$. Now the claim immediately follows.

As we have done in the case (I), by considering from the view point of ramification, 
we show below that 
$$
[\mbb{Q}_2(E[8]):\mbb{Q}_2]<32
$$ 
holds 
(of course this is a contradiction).
First we note that $\mbb{Q}_2(E[2])$ is a subfield of 
$\mbb{Q}_2(\mu_8)$ since $E(\mbb{Q}_2(\mu_8))$ contains $E[2]$.
By \eqref{u:relation1} and the fact that 
 $u\ \mrm{mod}\ 4$ on $G_{\mbb{Q}_2(E[2])}$ 
is a homomorphism with 
values in $2\cdot \mbb{Z}/4\mbb{Z}$,
we know that the homomorphism 
$u\ \mrm{mod}\ 4\colon G_{\mbb{Q}_2(E[2])}\to 2\cdot \mbb{Z}/4\mbb{Z}$ 
is surjective.
We denote by $L$ the quadratic extension of $\mbb{Q}_2(E[2])$
which corresponds to the kernel of this homomorphism.
By definition of $L$, $L$ is a quadratic extension of $\mbb{Q}_2(E[2])$
and we have 
\begin{equation}
\label{u:incl}
u(G_L)\subset 4\cdot \mbb{Z}/8\mbb{Z}.
\end{equation}
Since $L$ is a subfield of $\mbb{Q}_2(E[4])$, 
it follows from Fontaine's ramification bound \eqref{Fontaine:bound} that 
$u_{L/\mbb{Q}_2}\le 3$.
Furthermore, $L$ does not contained in $\mbb{Q}_2(\mu_8)$.
In fact, if $L$ is a subfield of  $\mbb{Q}_2(\mu_8)$, 
then $\rho_{E[4]}$ must be trivial on  $G_{\mbb{Q}_2(\mu_8)}$,
which shows $E(\mbb{Q}_2(\mu_8))$ contains $E[4]$ but this is a contradiction.
Since $\mbb{Q}_2(E[2])$ is now contained in $\mbb{Q}_2(\mu_8)$
and is of degree at most $2$ over $\mbb{Q}_2$,
it follows \eqref{Fontaine:bound} again that $\mbb{Q}_2(E[2])$ 
is either $\mbb{Q}_2$ or $\mbb{Q}_2(\mu_4)$
(see Table \ref{table:degree2}).
We make a case distinction depending on which of these two situations occurs.

(II-1) Suppose that $\mbb{Q}_2(E[2])=\mbb{Q}_2$. 
Since $\rho_{E[4]}=
\begin{pmatrix}
\chi_2 \ \mrm{mod}\ 4 & u\ \mrm{mod}\ 4\  \\
0 & 1
\end{pmatrix}$
on $G_{\mbb{Q}_2}$ by \eqref{psi:relation},
we have $\mbb{Q}_2(E[4])=L(\mu_4)$.
We recall that $L$ satisfies all the following properties:
\begin{enumerate}
\item[(a)] $L$ is a quadratic extension of $\mbb{Q}_2$,
\item[(b)] $u_{L/\mbb{Q}_2}\le 3$ and 
\item[(c)] $L$ does not contained in $\mbb{Q}_2(\mu_8)$.
\end{enumerate}
Note that $u_{L/\mbb{Q}_2}$ is not equal to one since $L/\mbb{Q}_2$ is 
either
unramified or wildly ramified.
\begin{itemize}
\item Suppose $u_{L/\mbb{Q}_2}=0$.
If this is the case, we have $L=F$ and 
$\rho_{E[8]}=\begin{pmatrix}
\chi_2 \ \mrm{mod}\ 8 & u  \\
0 & 1
\end{pmatrix}$
on $G_F$ by \eqref{psi:relation}.
Thus it follows from \eqref{u:incl} that we have 
$[\mbb{Q}_2(E[8]):F]\le 8$,
which shows $[\mbb{Q}_2(E[8]):\mbb{Q}_2]\le 16<32$
as desired.
\item Suppose $u_{L/\mbb{Q}_2}=2$.
By (a), (b) and (c) above, we find that 
$L$ is isomorphic to $L_2=\mbb{Q}_2[x]/(x^2+2x+6)$
(see Table \ref{table:degree2}).
In particular, $L(\mu_4)$ contains $F$ 
by Proposition \ref{composite} (1).
We have $\rho_{E[8]}=\begin{pmatrix}
\chi_2 \ \mrm{mod}\ 8 & u  \\
0 & 1
\end{pmatrix}$
on $G_{L(\mu_4)}$ by \eqref{psi:relation}.
Thus it follows from \eqref{u:incl} that
we have $[\mbb{Q}_2(E[8]):L(\mu_4)]\le 4$,
which shows $[\mbb{Q}_2(E[8]):\mbb{Q}_2]\le 16<32$
as desired.
\item Suppose $u_{L/\mbb{Q}_2}=3$.
There exist only $4$ possibility for such $L$.
Explicitly, $L$ is isomorphic to one of the following
(see Table \ref{table:degree2}).
\begin{align*}
& L_3=\mbb{Q}_2[x]/(x^2+2), 
&&L_4=\mbb{Q}_2[x]/(x^2+10), \\
&L_5=\mbb{Q}_2[x]/(x^2+4x+2), 
&&L_6=\mbb{Q}_2[x]/(x^2+4x+10). 
\end{align*}
For each of the fields, their composite with $F(\mu_4)$ contains $\mu_8$
by Proposition \ref{composite} (2).
This implies that  $LF(\mu_4)$ contains $\mu_8$.
Hence we have $\rho_{E[8]}=\begin{pmatrix}
1 & u\ \mrm{mod}\ 8\  \\
0 & 1
\end{pmatrix}$
on $G_{LF(\mu_4)}$ by \eqref{psi:relation}.
Thus it follows from \eqref{u:incl} that we have 
$[\mbb{Q}_2(E[8]):LF(\mu_4)]\le 2$,
which shows $[\mbb{Q}_2(E[8]):\mbb{Q}_2]\le 16<32$
as desired.
\end{itemize}

(II-2) Suppose that $\mbb{Q}_2(E[2])=\mbb{Q}_2(\mu_4)$.
Since $\rho_{E[4]}=
\begin{pmatrix}
1 & u\ \mrm{mod}\ 4\  \\
0 & 1
\end{pmatrix}$
on $G_{\mbb{Q}_2(\mu_4)}$ by \eqref{psi:relation},
we have $\mbb{Q}_2(E[4])=L$. 
In particular, $L$ is a Galois extension of $\mbb{Q}_2$.
Here we summarize properties of $L$.
\begin{enumerate}
\item[(a)] $L$ is a Galois extension over $\mbb{Q}_2$ of degree $4$
with $L\supset \mbb{Q}_2(\mu_4)$,
\item[(b)] $u_{L/\mbb{Q}_2}\le 3$ and 
\item[(c)] $L$ does not conatined in $\mbb{Q}_2(\mu_8)$.
\end{enumerate}
We claim that either $L=F(\mu_4)$
or $LF\supset \mbb{Q}_2(\mu_8)$.
If $L/\mbb{Q}_2(\mu_4)$ is unramified, then we have $L=F(\mu_4)$.
Suppose that $L/\mbb{Q}_2(\mu_4)$ is totally ramified.
Then $L$ is a totally ramified Galois extension over $\mbb{Q}_2$
of degree $4$.
By $\mbb{Q}_2(\mu_4)\subset L=\mbb{Q}_2(E[4])$, Fontiane's ramification bound implies 
$2\le u_{L/\mbb{Q}_2}\le 3$.
In fact, there does not exist\footnote{
This is easily checked as in the database of $p$-adic fields in 
LMFDB \cite{lmfdb}. 
}
a totally ramified Galois extension 
over $\mbb{Q}_2$ of degree $4$ 
with maximal ramification break $2$.
Thus we have $u_{L/\mbb{Q}_2}=3$.
There exist only $4$ possibilities of such $L$;
$L$ is isomorphic to one of the followings
(see Table \ref{table:degree4}).
\begin{align*}
& M_1=\mbb{Q}_2[x]/(x^4+2x^2+4x+2), 
&&M_2=\mbb{Q}_2[x]/(x^4+2x^2+4x+10), \\
&M_3=\mbb{Q}_2[x]/(x^4+4x^3+2x^2+4x+6), 
&&M_4=\mbb{Q}_2[x]/(x^4+4x^3+2x^2+4x+14). 
\end{align*}
In each case\footnote{In fact, since $L$ contains $\mu_4$ but does not contain $\mu_8$,
we obtain the fact that $L=M_2$.}, we know that $LF$ contains $\mu_8$
by Proposition \ref{composite} (3). Thus the claim follows.
\begin{itemize}
\item Suppose $L=F(\mu_4)$.
We have $\rho_{E[8]}=\begin{pmatrix}
1 & u  \\
0 & 1
\end{pmatrix}$
on $G_{F(\mu_8)}$ by \eqref{psi:relation}.
Thus it follows from \eqref{u:incl} that we have 
$[\mbb{Q}_2(E[8]):F(\mu_8)]\le 2$,
which shows $[\mbb{Q}_2(E[8]):\mbb{Q}_2]\le 16<32$
as desired.
\item Suppose $LF\supset \mbb{Q}_2(\mu_8)$.
We have $\rho_{E[8]}=\begin{pmatrix}
1 & u  \\
0 & 1
\end{pmatrix}$
on $G_{LF}$ by \eqref{psi:relation}.
Thus it follows from \eqref{u:incl} that we have 
$[\mbb{Q}_2(E[8]):LF]\le 2$,
which shows $[\mbb{Q}_2(E[8]):\mbb{Q}_2]\le 16<32$
as desired.
\end{itemize}
Therefore, we finish the proof of the theorem.
\end{proof}    

\begin{remark}
\label{remark:nonexistence}
As we have seen in the arguments of (I) and (II) above, it holds that  
$E(\mbb{Q}_2(\mu_4))_{\mrm{tor}}\not \simeq \mbb{Z}/2\mbb{Z}$
and 
$E(\mbb{Q}_2(\mu_8))_{\mrm{tor}}\not \simeq \mbb{Z}/2\mbb{Z}\times \mbb{Z}/4\mbb{Z}$
for any elliptic curve $E$ over $\mbb{Q}_2$ with good ordinary reduction.
\end{remark}

Next we show the first statement of Theorem \ref{E(Qp(mu)):good} (3). 

\begin{proof}[Proof of the first statement of 
Theorem \ref{E(Qp(mu)):good} (3)]
By Fontaine's ramification bound \eqref{Fontaine:bound},  
any element of $E(\mbb{Q}_2(\mu_{2^{\infty}}))_{\mrm{tor}}$ killed by $2^n$
is rational over $\mbb{Q}_2(\mu_{n+1})$
for any $n$.
Thus, by the second statement of Theorem \ref{E(Qp(mu)):good} (3),
which has already been proved, 
it suffices to show 
$E(\mbb{Q}_2(\mu_{2^{\infty}}))[2^{\infty}]
=E(\mbb{Q}_2(\mu_8))[2^{\infty}]$ 
if 
$E(\mbb{Q}_2(\mu_{2^{\infty}}))[2^{\infty}]$ is isomorphic to 
either $\mbb{Z}/8\mbb{Z}$ or $\mbb{Z}/2\mbb{Z}\times \mbb{Z}/8\mbb{Z}$.
If this is the case, 
by Lemma \ref{hat=bar} and the condition $\# E(\mbb{Q}_2(\mu_{2^{\infty}}))[2^{\infty}]>4$,
we have $\bar{E}(\mbb{F}_2)\simeq \mbb{Z}/4\mbb{Z}$.
Thus the $G_{\mbb{Q}_2}$-action on $E[4]$ is given by 
\begin{equation}
\label{E4}
\rho_{E[4]}=\begin{pmatrix}
\chi_2\ \mrm{mod}\ 4 & u  \\
0 & 1
\end{pmatrix}
\colon G_{\mbb{Q}_2}\to GL_2(\mbb{Z}/4\mbb{Z})
\end{equation}
for some map $u\colon G_{\mbb{Q}_2}\to \mbb{Z}/4\mbb{Z}$.

First we consider the case where 
$E(\mbb{Q}_2(\mu_{2^{\infty}}))[2^{\infty}]\simeq \mbb{Z}/8\mbb{Z}$. 
Since $E(\mbb{Q}_2(\mu_{2^{\infty}}))[2^{\infty}]$ does not contain $E[2]$,
it holds $u \ \mrm{mod}\ 2\not=0$ on $G_K$. 
In particular, 
$\rho_{E[4]}(G_{\mbb{Q}_2(\mu_4)})$
is generated by 
$
\begin{pmatrix}
1 & 1\\
0 & 1
\end{pmatrix}.
$
We assume that $\mbb{Q}_2(E[4])\supset \mbb{Q}_2(\mu_8)$.
If this is the case,  
$\rho_{E[4]}(G_{\mbb{Q}_2(\mu_8)})$
is generated by 
$
\begin{pmatrix}
1 & 2\\
0 & 1
\end{pmatrix}
$
since it is the subgroup of index $2$ in 
$\rho_{E[4]}(G_{\mbb{Q}_2(\mu_4)})$.
Thus we see that $E(\mbb{Q}_2(\mu_8))$ contains $\mbb{Z}/2\mbb{Z}\times \mbb{Z}/4\mbb{Z}$
but this contradicts the assumption $E(\mbb{Q}_2(\mu_{2^{\infty}}))[2^{\infty}]\simeq \mbb{Z}/8\mbb{Z}$. 
Thus we obtain $\mbb{Q}_2(E[4])\not\supset \mbb{Q}_2(\mu_8)$.
This is equivalent to say that  
$\mbb{Q}_2(E[4])\cap \mbb{Q}_2(\mu_{2^{\infty}})=
\mbb{Q}_2(\mu_4)$.
Now we claim\footnote{This claim follows immediately from 
Proposition 2.11 of \cite{BrKr01}. Our proof of the claim here  
follows their arguments.} 
$$
\mbb{Q}_2(E[8])\cap \mbb{Q}_2(\mu_{2^{\infty}})=
\mbb{Q}_2(\mu_8).
$$
Consider a homomorphism
$\mrm{Gal}(\mbb{Q}_2(E[8])/\mbb{Q}_2)
\to GL_2(\mbb{Z}/8\mbb{Z})$
coming from the $G_{\mbb{Q}_2}$-action on $E[8]$.
Since the restriction of this map 
to $\mrm{Gal}(\mbb{Q}_2(E[8])/\mbb{Q}_2(E[4]))$
has values in the kernel of the mod $4$ reduction 
$GL_2(\mbb{Z}/8\mbb{Z})\to GL_2(\mbb{Z}/4\mbb{Z})$,
it holds that 
the Galois group of the extension 
$\mbb{Q}_2(E[8])$ over $\mbb{Q}_2(E[4])$ 
is of exponent $2$.
Thus 
the Galois group of the extension 
$\mbb{Q}_2(E[8])\cap \mbb{Q}_2(\mu_{2^{\infty}})$
over $\mbb{Q}_2(E[4])\cap \mbb{Q}_2(\mu_{2^{\infty}})$
is also of exponent $2$.
Let $n\ge 3$ be an integer such that 
$\mbb{Q}_2(E[8])\cap \mbb{Q}_2(\mu_{2^{\infty}})=\mbb{Q}_2(\mu_{2^n})$.
By the condition $\mbb{Q}_2(E[4])\cap \mbb{Q}_2(\mu_{2^{\infty}})=
\mbb{Q}_2(\mu_4)$,
the Galois group of the extension 
$\mbb{Q}_2(E[8])\cap \mbb{Q}_2(\mu_{2^{\infty}})$ over 
$\mbb{Q}_2(E[4])\cap \mbb{Q}_2(\mu_{2^{\infty}}))$ is isomorphic to 
$\mbb{Z}/2^{n-2}\mbb{Z}$.
Since this is of exponent $2$,
we have $n=3$. Thus the claim is now proved. 
Since 
$E(\mbb{Q}_2(\mu_{2^{\infty}}))[2^{\infty}]$
is killed by $8$, it follows from the claim that we have 
$$
E(\mbb{Q}_2(\mu_{2^{\infty}}))[2^{\infty}]
=E(\mbb{Q}_2(E[8])\cap \mbb{Q}_2(\mu_{2^{\infty}}))[2^{\infty}]
=E(\mbb{Q}_2(\mu_8))[2^{\infty}]
$$ 
as desired.

Next we consider the case where 
$E(\mbb{Q}_2(\mu_{2^{\infty}}))[2^{\infty}]\simeq \mbb{Z}/2\mbb{Z}\times \mbb{Z}/8\mbb{Z}$.
Assume that $E(\mbb{Q}_2(\mu_{2^{\infty}}))[2^{\infty}]\not= 
E(\mbb{Q}_2(\mu_{8}))[2^{\infty}]$.
Since $\mbb{Q}_2(E[2])$ is a subfield of $\mbb{Q}_2(\mu_{2^{\infty}})$,
it follows that $\mbb{Q}_2(E[2])$ is either $\mbb{Q}_2$ or $\mbb{Q}_2(\mu_4)$.
In particular, we have $\mbb{Q}_2(E[2])\subset \mbb{Q}_2(\mu_8)$.
Hence 
$E(\mbb{Q}_2(\mu_{8}))[2^{\infty}]$ is isomorphic to either 
$\mbb{Z}/2\mbb{Z}\times \mbb{Z}/2\mbb{Z}$ or $\mbb{Z}/2\mbb{Z}\times \mbb{Z}/4\mbb{Z}$.
By Remark \ref{remark:nonexistence},
we have $E(\mbb{Q}_2(\mu_{8}))[2^{\infty}]\simeq \mbb{Z}/2\mbb{Z}\times \mbb{Z}/2\mbb{Z}$.
However, \eqref{E4} implies that $E(\mbb{Q}_2(\mu_{8}))[4]$ must contain 
some element of order just $4$. This is a contradiction.
Therefore, we conclude $E(\mbb{Q}_2(\mu_{2^{\infty}}))[2^{\infty}]= 
E(\mbb{Q}_2(\mu_{8}))[2^{\infty}]$ as desired.
\end{proof}

\begin{remark}
We should note that the field $\mbb{Q}_p(\mu_{p})$ 
(resp.\ $\mbb{Q}_2(\mu_8)$) in the first statement of Theorem \ref{E(Qp(mu)):good} (2)
(resp.\ Theorem \ref{E(Qp(mu)):good} (3))
is  ``the best possible" 
in the sense that, 
for each odd prime $p$ (resp.\ $p=2$), 
there exists an elliptic cureve $E$ over $\mbb{Q}_p$
with good ordinary reduction such that 
the definition field of 
$E(\mbb{Q}_p(\mu_{p^{\infty}}))_{\mrm{tor}}$
is 
$\mbb{Q}_p(\mu_p)$ (resp.\ $\mbb{Q}_2(\mu_8)$).
Examples for such situations are as follows.

\begin{itemize}
\item[--] Suppose $p\ge 3$ and take an elliptic curve $\bar{E}$
such that $\bar{E}(\mbb{F}_p)\not=0$ (such $\bar{E}$ exists for any $p$).
Let $E_{/\mbb{Q}_p}$ be the canonical lift of $\bar{E}$.
Then, $E[p]$ is isomorphic to $\mbb{F}_p(1)\oplus \mbb{F}_p$ as a $G_{\mbb{Q}_p}$-representation
(see the argument of Section \ref{Qp:ord:odd}),
where $G_{\mbb{Q}_p}$ is the absolute Galois group of $\mbb{Q}_p$.
Thus $\mbb{Q}_p(E[p])=\mbb{Q}_p(\mu_p)$. On the other hand,
the prime-to-$p$ part of  $E(\mbb{Q}_p(\mu_{p^{\infty}}))$ 
coincides with that of  $E(\mbb{Q}_p)$ by the N\'eron-Ogg -Shafarevich criterion.
Thus, the definition field of $E(\mbb{Q}_p(\mu_{p^{\infty}}))_{\mrm{tor}}$
is $\mbb{Q}_p(\mu_p)$.
\item[--] 
The definition field of 
$E(\mbb{Q}_2(\mu_{2^{\infty}}))_{\mrm{tor}}$
for $E=$\href{https://www.lmfdb.org/EllipticCurve/Q/15/a/4}{15.a7} 
is $\mbb{Q}_2(\mu_8)$.
\end{itemize}
\end{remark}

We end this section by proving Theorem \ref{E(Q2(mu4)):good}.

\begin{proof}[Proof of Theorem \ref{E(Q2(mu4)):good}]
Since the group $E(\mbb{Q}_2(\mu_4))_{\mrm{tor}}$ is a subgroup of 
$E(\mbb{Q}_2(\mu_{2^{\infty}}))_{\mrm{tor}}$,
it is isomorphic to a   subgroup 
of the groups listed in $(I)'_{\infty}$.
Since the kernel of the reduction map $E[2]\to \bar{E}[2]$
is rational over $\mbb{Q}_2$, we know that 
$E(\mbb{Q}_2)_{\mrm{tor}}$ 
(and also $E(\mbb{Q}_2(\mu_4))_{\mrm{tor}}$) are not zero.
Furthermore, by Remark \ref{remark:nonexistence},
we also have 
$E(\mbb{Q}_2(\mu_4))_{\mrm{tor}}\not \simeq \mbb{Z}/2\mbb{Z}$.
On the other hand, 
for each group $G_{i,j}$ in $(I)'_2$, MAGMA computations with Algorithm \ref{alg:torsion} 
give examples 
of $E$ such that $E(\mbb{Q}_2(\mu_4))_{\mrm{tor}}\simeq G_{i,j}$ as follows.
\begin{itemize}
\item 
$E=$\href{https://www.lmfdb.org/EllipticCurve/Q/15/a/1}{15.a1} satisfies $E(\mbb{Q}_2(\mu_4))_{\mrm{tor}}\simeq G_{1,4}$,
\item 
$E=$\href{https://www.lmfdb.org/EllipticCurve/Q/33/a/2}{33.a2} satisfies $E(\mbb{Q}_2(\mu_4))_{\mrm{tor}}\simeq G_{2,1}$,
\item 
$E=$\href{https://www.lmfdb.org/EllipticCurve/Q/15/a/2}{15.a2} satisfies $E(\mbb{Q}_2(\mu_4))_{\mrm{tor}}\simeq G_{2,2}$,
\item 
$E=$\href{https://www.lmfdb.org/EllipticCurve/Q/15/a/8}{15.a8} satisfies $E(\mbb{Q}_2(\mu_4))_{\mrm{tor}}\simeq G_{2,4}$,
\item 
$E=$\href{https://www.lmfdb.org/EllipticCurve/Q/15/a/5}{15.a5} satisfies $E(\mbb{Q}_2(\mu_4))_{\mrm{tor}}\simeq G_{4,1}$.
\end{itemize}
Therefore, to finish the proof of the theorem,
it suffices to show that 
$E(\mbb{Q}_2(\mu_4))_{\mrm{tor}}$ is not isomorphic to 
$G_{1,8}=\mbb{Z}/8\mbb{Z}$.
Assume $E(\mbb{Q}_2(\mu_4))_{\mrm{tor}}\simeq \mbb{Z}/8\mbb{Z}$.
By Theorem \ref{E(Qp(mu)):good} (3), 
$E(\mbb{Q}_2(\mu_8))_{\mrm{tor}}$
is isomorphic to $\mbb{Z}/8\mbb{Z}$ or $\mbb{Z}/2\mbb{Z}\times \mbb{Z}/8\mbb{Z}$.
By assumption, we have $E(\mbb{Q}_2(\mu_4))_{\mrm{tor}}\not \supset E[2]$
and thus $\mbb{Q}_2(E[2])$ is not contained in $\mbb{Q}_2(\mu_4)$.
Since the extension $\mbb{Q}_2(E[2])/\mbb{Q}_2$ is of degree at most 
$2$ and $u_{\mbb{Q}_2(E[2])/\mbb{Q}_2}\le 2$,
it holds that 
$\mbb{Q}_2(E[2])$ is $F$ or $L_2$
(see Table \ref{table:degree2}).
In both cases, $\mbb{Q}_2(E[2])$ is not a subfield of $\mbb{Q}_2(\mu_8)$.
Thus $E(\mbb{Q}_2(\mu_8))_{\mrm{tor}}$ does not contain $E[2]$.
Therefore, we obtain 
$$
E(\mbb{Q}_2(\mu_4))_{\mrm{tor}}=E(\mbb{Q}_2(\mu_8))_{\mrm{tor}}\simeq \mbb{Z}/8\mbb{Z}.
$$

Let $\chi\colon G_{\mbb{Q}_2}\to \mbb{Z}_2^{\times}$
and $\psi\colon G_{\mbb{Q}_2}\to \mbb{Z}_2^{\times}$
be the crystalline characters obtained by 
the $G_{\mbb{Q}_2}$-action on the $2$-adic Tate modules $T_2(\hat{E})$ and $T_2(\bar{E})$,
respectively.
Consider the $G_{\mbb{Q}_2}$-action on $E[8]$.
Identifying $E[8]=\mbb{Z}/8\mbb{Z}\times \mbb{Z}/8\mbb{Z}$ 
by a suitable choice of $\mbb{Z}/8\mbb{Z}$-basis of $E[8]$, 
 the $G_{\mbb{Q}_2}$-action on $E[8]$ is given by a continuous homomorphism 
$\rho_{E[8]}\colon G_{\mbb{Q}_2}\to GL_2(\mbb{Z}/8\mbb{Z})$
with the matrix form 
\begin{equation}
\label{E8}
\rho_{E[8]}=\begin{pmatrix}
\chi\ \mrm{mod}\ 8 & u  \\
0 & \psi \ \mrm{mod}\ 8
\end{pmatrix}
\end{equation}
for some map $u\colon G_{\mbb{Q}_2}\to \mbb{Z}/8\mbb{Z}$.
Let us study subgroups 
$\rho_{E[8]}(G_{\mbb{Q}_2(\mu_8)})$
and 
$\rho_{E[8]}(G_{\mbb{Q}_2(\mu_4)})$
of $GL_2(\mbb{Z}/8\mbb{Z})$.
Since we have $\bar{E}(\mbb{F}_2)\simeq \mbb{Z}/4\mbb{Z}$
by Lemma \ref{hat=bar}, 
the image of $\psi \ \mrm{mod}\ 8$ on $G_{\mbb{Q}_2}$ 
is $\{1,5\} \subset (\mbb{Z}/8\mbb{Z})^{\times}$.
In particular,
$\psi \ \mrm{mod}\ 8$ is an unramified character with kernel $G_F$.
Since $\chi \psi \ \mrm{mod}\ 8$ is trivial on $G_{\mbb{Q}_2(\mu_8)}$, 
we see $\chi \equiv \psi\ \mrm{mod}\ 8$ on $G_{\mbb{Q}_2(\mu_8)}$.
We find that 
$\rho_{E[8]}(G_{\mbb{Q}_2(\mu_8)})$ is a subgroup of  
$$
H:=\left\{ 
\begin{pmatrix}
a & b \\
0 & a 
\end{pmatrix}
\mid a=1,5\in \mbb{Z}/8\mbb{Z}\  \mbox{and}\ b\in \mbb{Z}/8\mbb{Z}
\right\}
$$
The order of the group $H$ is 16.
We note that $u \ \mrm{mod} \ 2$ is not trivial on $G_{\mbb{Q}_2(\mu_8)}$
since $E(\mbb{Q}_2(\mu_8))_{\mrm{tor}}$ does not contain $E[2]$.
Thus there exist $a\in \mbb{Z}/8\mbb{Z}$ 
and an odd $b\in \mbb{Z}/8\mbb{Z}$ 
such that $\begin{pmatrix}
a & b \\
0 & a 
\end{pmatrix}\in \rho_{E[8]}(G_{\mbb{Q}_2(\mu_8)})$.
Since the group generated by such a matrix 
$\begin{pmatrix}
a & b \\
0 & a 
\end{pmatrix}$
is either the group $H_1$ generated by
$\begin{pmatrix}
1 & 1 \\
0 & 1 
\end{pmatrix}$  or
the group $H_2$ generated by 
$\begin{pmatrix}
5 & 1 \\
0 & 5 
\end{pmatrix}$, 
we find that 
$\rho_{E[8]}(G_{\mbb{Q}_2(\mu_8)})$
is either 
$H_1$,\ $H_2$ or $H$.
By the facts that  
$\psi \ \mrm{mod} \ 8$ is not trivial on $G_{\mbb{Q}_2(\mu_8)}$
and $E(\mbb{Q}_2(\mu_8))_{\mrm{tor}}\simeq \mbb{Z}/8\mbb{Z}$, 
we obtain that $\rho_{E[8]}(G_{\mbb{Q}_2(\mu_8)})$
is not equal to $H_1$ and $H$, that is, 
\begin{equation}
\label{rho8mu8}
\rho_{E[8]}(G_{\mbb{Q}_2(\mu_8)})
=H_2=\langle \begin{pmatrix}
5 & 1 \\
0 & 5 
\end{pmatrix}\rangle
\end{equation}
and 
$E(\mbb{Q}_2(\mu_8))[2^{\infty}]
=\langle \begin{pmatrix}
1  \\
4  
\end{pmatrix}\rangle.
$
Next we consider $\rho_{E[8]}(G_{\mbb{Q}_2(\mu_4)})$.
We find that 
$\rho_{E[8]}(G_{\mbb{Q}_2(\mu_4)})$ is a subgroup of  
$$
G:=\left\{ 
\begin{pmatrix}
a & b \\
0 & c 
\end{pmatrix}
\mid a,c=1,5\in \mbb{Z}/8\mbb{Z}\  \mbox{and}\ b\in \mbb{Z}/8\mbb{Z}
\right\}
$$
by \eqref{E8} and the fact that $\psi$ and $\chi_2$ on $G_{\mbb{Q}_2(\mu_4)}$ 
have values in $\{1,5\}\subset (\mbb{Z}/8\mbb{Z})^{\times}$.
The order of the group $G$ is 32.
We see that 
$\rho_{E[8]}(G_{\mbb{Q}_2(\mu_8)})$
is a normal subgroup of $G$, the quotient 
$G/\rho_{E[8]}(G_{\mbb{Q}_2(\mu_8)})$
is isomorphic to $\mbb{Z}/2\mbb{Z}\times \mbb{Z}/2\mbb{Z}$
and $G/\rho_{E[8]}(G_{\mbb{Q}_2(\mu_8)})$
is generated by the classes with respect to 
$\begin{pmatrix}
5 & 0\\
0 & 1
\end{pmatrix}$
and 
$\begin{pmatrix}
1 & 0\\
0 & 5
\end{pmatrix}$.
Since 
the quotient $\rho_{E[8]}(G_{\mbb{Q}_2(\mu_4)})/\rho_{E[8]}(G_{\mbb{Q}_2(\mu_8)})$
is a subgroup of $G/\rho_{E[8]}(G_{\mbb{Q}_2(\mu_8)})$
of order $2$,
we find that $\rho_{E[8]}(G_{\mbb{Q}_2(\mu_4)})$
is one of the following groups; 
$$
\langle 
\begin{pmatrix}
5 & 0 \\
0 & 1 
\end{pmatrix},
\begin{pmatrix}
5 & 1 \\
0 & 5 
\end{pmatrix}\rangle,\quad 
\langle 
\begin{pmatrix}
1 & 0 \\
0 & 5 
\end{pmatrix},
\begin{pmatrix}
5 & 1 \\
0 & 5 
\end{pmatrix}\rangle,\quad 
\langle 
\begin{pmatrix}
5 & 0 \\
0 & 5 
\end{pmatrix}, 
\begin{pmatrix}
5 & 1 \\
0 & 5 
\end{pmatrix}\rangle.
$$
Since we now have 
$E(\mbb{Q}_2(\mu_4))[2^{\infty}]
=E(\mbb{Q}_2(\mu_8))[2^{\infty}]
=\langle \begin{pmatrix}
1  \\
4  
\end{pmatrix}\rangle$,
we find 
\begin{equation}
\label{rho8mu4}
\rho_{E[8]}(G_{\mbb{Q}_2(\mu_4)})
=\langle \begin{pmatrix}
1 & 0 \\
0 & 5 
\end{pmatrix},
\begin{pmatrix}
5 & 1 \\
0 & 5 
\end{pmatrix}\rangle.
\end{equation}
It follows from \eqref{rho8mu8} and  \eqref{rho8mu4}
that there exists 
$\sigma_0\in G_{\mbb{Q}_2(\mu_4)}
\smallsetminus G_{\mbb{Q}_2(\mu_8)}$ 
such that 
$\rho_{E[8]}(\sigma_0)
=\begin{pmatrix}
1 & 0 \\
0 & 5 
\end{pmatrix}$.
Since $\psi(\sigma_0)\equiv 5$ mod $8$, 
we have 
\begin{equation}
\label{sigma0}
\sigma_0\not\in G_F.
\end{equation}
On the other hand, as we have already checked, 
the field $\mbb{Q}_2(E[2])$
is either $F$ or $L_2$ (see Table \ref{table:degree2}).
In each case, 
$\mbb{Q}_2(E[2],\mu_4)$ contains $F$
by Proposition \ref{composite} (1).
Since $\rho_{E[8]}(\sigma_0)$ mod $2$ is trivial, 
we have $\sigma_0\in G_{\mbb{Q}_2(E[2])}$,
which gives 
$\sigma_0\in G_{\mbb{Q}_2(E[2])}\cap G_{\mbb{Q}_2(\mu_4)}
\subset G_F$.
This contradicts \eqref{sigma0}.
Therefore, we conclude $E(\mbb{Q}_2(\mu_4))\not\simeq \mbb{Z}/8\mbb{Z}$
as desired and this finishes the proof.
\end{proof}


\appendix

\section{Appendix : Data and algorithm}
\label{Data}

In this section, we present the data obtained by using the computer algebra system MAGMA \cite{BoCa06}.  

\subsection{Quadratic and quartic extensions of $\mbb{Q}_2$}

In this subsection, we give the tables of quadratic and 
certain quartic extensions of $\mbb{Q}_2$ for 
using to our proof in subsection \ref{subsection:p=2}. 
They can be easily checked by LMFDB database \cite{lmfdb}.  
Tabel \ref{table:degree2} shows the defining polynomials for 
all the quadratic extensions of $\mbb{Q}_2$. 
Tabel \ref{table:degree4} shows the defining polynomials for 
all the quartic extensions $L/\mbb{Q}_2$ with 
$u_{L/\mbb{Q}_2}=3$.
In these tables, the integer $f$ presents the inertia degree, $e$ the ramification index, 
$u_{L/\mbb{Q}_2}$ the maximal upper ramification break and $\mu_n$ the roots of unity which 
are included in $L$. 

\begin{table}[H]
\begin{center}
\begin{tabular}{|c|c|c|c|c|c|} \hline 
     $L$ & Polynomial & $f$ & $e$ & $u_{L/\mbb{Q}_2}$ & $\mu_n$ \\ \hline
  $F$ & $x^2+x+1$ & $2$ & $1$ & $0$ & $\mu_6$ \\ \hline
  $L_1$ & $x^2+2x+2$ & $1$ & $2$ & $2$ & $\mu_4$ \\ \hline
  $L_2$ & $x^2+2x+6$ & $1$ & $2$ & $2$ & $\mu_2$ \\ \hline
  $L_3$ & $x^2+2$ & $1$ & $2$ & $3$ & $\mu_2$ \\ \hline
  $L_4$ & $x^2+10$ & $1$ & $2$ & $3$ & $\mu_2$ \\ \hline
  $L_5$ & $x^2+4x+2$ & $1$ & $2$ & $3$ & $\mu_2$ \\ \hline
  $L_6$ & $x^2+4x+10$ & $1$ & $2$ & $3$ & $\mu_2$ \\ \hline
\end{tabular}
\end{center}
\caption{The quadratic extensions of $\mbb{Q}_2$}\label{table:degree2}
\end{table}

\begin{table}[H]
\begin{center}
\begin{tabular}{|c|c|c|c|c|c|} \hline 
 $L$ & Polynomial & $f$ & $e$ & $u_{L/\mbb{Q}_2}$ & $\mu_n$  \\ \hline
  $M_1$ & $x^4+2x^2+4x+2$ & $1$ & $4$ & $3$ & $\mu_8$ \\ \hline
  $M_2$ & $x^4+2x^2+4x+10$ & $1$ & $4$ & $3$ & $\mu_4$ \\ \hline
  $M_3$ & $x^4+4x^3+2x^2+4x+6$ & $1$ & $4$ & $3$ & $\mu_2$ \\ \hline
  $M_4$ & $x^4+4x^3+2x^2+4x+14$ & $1$ & $4$ & $3$ & $\mu_2$ \\ \hline
\end{tabular}
\end{center}
\caption{The quartic Galois extensions of $\mbb{Q}_2$ with $u_{L/\mbb{Q}_2}=3$}
\label{table:degree4}
\end{table}

We can easily check the equalities of the composite fields by using MAGMA as follows:

\begin{proposition}
\label{composite}
In Table \ref{table:degree2} and \ref{table:degree4}, we have the equalities:
    \begin{itemize}    
    \item[\rm{(1)}] $L_1 L_2 = F(\mu_4)$. 

    \item[\rm{(2)}] $L_3F(\mu_4)=L_4F(\mu_4)=L_5F(\mu_4)=L_6F(\mu_4)=F(\mu_8)$.
    
    \item[\rm{(3)}] $FM_1=FM_2=FM_3=F M_4=F(\mu_8)$. 
    \end{itemize}
\end{proposition}


\subsection{Algorithm for computing $\# E(K)[n]$}
In this subsection, we give an algorithm for calculating the order $\# E(K)[n]$ 
for given $K$, $E$ and $n$, where $K$ is a finite extension of $\mathbb{Q}_p$, 
$E$ is an elliptic curve over $K$ and $n>1$ is an integer. 
In the case where $K$ is an algebraic number filed, 
we can compute the torsion subgroup $E(K)_{\mrm{tor}}$ by the 
intrinsic function 
$\texttt{TorsionSubgroup}(E)$ which is implemented in MAGMA\footnote{
The version of MAGMA used in this paper is v2.27-5.
}. 
However, this intrinsic function is not valid for a $p$-adic field $K$. 
In spite of such a function, we use the several functions that 
they are implemented in MAGMA and valid even for a $p$-adic field. 
First, the function $\texttt{DivisionPolynomial}(E,n)$ gives a 
polynomial whose roots are the $x$-coordinates of the points of $E(K)[n]$ 
for a $p$-adic field $K$, 
an elliptic curve $E$ and an integer $n>1$. 
Second, the function $\texttt{Roots}(f)$ gives the roots of a given 
polynomial $f$ over $K$. 
Finally, the function $\texttt{Points}(E(K),x_0)$ gives 
the points of $E$ whose $x$-coordinate equals $x_0$.
By combining these functions, we can compute 
the order $\# E(K)[n]$ as follows: 

\begin{algorithm}[H]
	\label{alg1}
	\begin{algorithmic}
	\REQUIRE $n>1$, $K$, $E$
    \ENSURE $t=\# E(K)[n]$
    \STATE $t \leftarrow 1$
    \STATE $f \leftarrow \texttt{DivisionPolynomial}(E,n)$
    \STATE $R \leftarrow \texttt{Roots}(f,K)$
    \STATE $r \leftarrow \# R$
    \IF{$r \not= 0$}
    \FOR{$x \in R$}
    \STATE $t \leftarrow t+\# \texttt{Points}(E(K),x)$
    \ENDFOR
    \ENDIF
    \RETURN $t$
	\end{algorithmic}\caption{Calculate $\# E(K)[n]$}\label{alg:torsion}
\end{algorithm}

\subsection{List of torsion subgroups}

In this subsection, we give a list of possible candidates that they actually occur 
for torsion subgroups of $E(\mbb{Q}_p)$ and $E(\mbb{Q}_p(\mu_{p^{\infty}}))$. 
In each list, we use Cremona's database of elliptic curves. 
Hence the label of each elliptic curve is presented as Cremona label, which is 
different from the one used in LMFDB.
We use Algorithm \ref{alg:torsion} which can compute $\# E(K)[n]$ 
for given $K$, $E$ and $n>1$. 
As in the proof of Theorem \ref{E(Qp(mu)):good}, 
it is enough to compute 
$E(\mbb{Q}_p(\mu_p))$ (resp. $E(\mbb{Q}_2(\mu_{16})))$ instead of computing 
$E(\mbb{Q}_p(\mu_{p^{\infty}}))$ for an odd prime $p$ (resp. for $p=2)$\footnote{
In case $p=2$, we actually have $E(\mbb{Q}_2(\mu_{2^{\infty}}))_{\mrm{tor}}
=E(\mbb{Q}_2(\mu_8))_{\mrm{tor}}$, as stated in Theorem \ref{E(Qp(mu)):good}. 
}.

\begin{table}[H]
\begin{center}
\begin{tabular}{|c|c|} \hline
   $E(\mbb{Q}_2)$ & Label \\ \hline 
   $G_{1,2}$ & 15a5 \\ 
   $G_{1,4}$ & 15a7 \\ 
   $G_{1,8}$ & 15a4 \\ 
   $G_{2,1}$ & 15a2 \\ 
   $G_{2,2}$ & 15a1 \\ 
   & \\ 
   & \\
   & \\
   & \\
   & \\
   & \\
   & \\
   & \\
   & \\
   \hline
\end{tabular} \hspace{3pt}
\begin{tabular}{|c|c|} \hline
   $E(\mbb{Q}_3)$ & Label \\ \hline 
   $G_{1,1}$ & 26a2 \\ 
   $G_{1,2}$ & 14a3 \\ 
   $G_{1,3}$ & 26a1 \\ 
   $G_{1,5}$ & 11a1 \\ 
   $G_{1,6}$ & 14a1 \\ 
   & \\ 
   & \\
   & \\
   & \\
   & \\
   & \\
   & \\
   & \\
   & \\
   \hline
\end{tabular} \hspace{3pt}
\begin{tabular}{|c|c|} \hline
   $E(\mbb{Q}_5)$ & Label \\ \hline 
   $G_{1,1}$ & 11a2 \\
   $G_{1,2}$ & 38b2 \\ 
   $G_{1,3}$ & 19a1 \\ 
   $G_{1,4}$ & 39a2 \\ 
   $G_{1,5}$ & 11a1 \\ 
   $G_{1,7}$ & 26b1 \\ 
   $G_{1,8}$ & 17a3 \\
   $G_{1,9}$ & 26a1 \\
   $G_{1,10}$ & 38b1 \\
   $G_{2,1}$ & 39a1 \\
   $G_{2,2}$ & 17a1 \\ 
   & \\
   & \\ 
   & \\ \hline
\end{tabular} \hspace{3pt}
\begin{tabular}{|c|c|} \hline
   $E(\mbb{Q}_7)$ & Label \\ \hline 
   $G_{1,1}$ & 26b2 \\
   $G_{1,3}$ & 104a1 \\ 
   $G_{1,4}$ & 17a1 \\ 
   $G_{1,5}$ & 38b1 \\ 
   $G_{1,6}$ & 20a1 \\ 
   $G_{1,7}$ & 26b1 \\ 
   $G_{1,9}$ & 19a2 \\
   $G_{1,10}$ & 11a1 \\
   $G_{1,11}$ & 75a1 \\
   $G_{1,12}$ & 30a1 \\
   $G_{1,13}$ & 57a1 \\
   $G_{2,1}$ & 17a2 \\ 
   $G_{2,3}$ & 30a2 \\
   $G_{3,1}$ & 19a1 \\ \hline
\end{tabular}
\caption{Examples of $E(\mbb{Q}_p)_{\mrm{tor}}$ 
with good ordinary reduction}\label{table:ord}
\end{center}
\end{table}

\begin{table}[H]
\begin{center}
\begin{tabular}{|c|c|} \hline
   $E(\mbb{Q}_2)$ & Label \\ \hline
   $G_{1,1}$ & 67a1 \\
   $G_{1,3}$ & 19a1 \\
   $G_{1,5}$ & 11a1 \\
   & \\ \hline
\end{tabular} \hspace{3pt}
\begin{tabular}{|c|c|} \hline
   $E(\mbb{Q}_3)$ & Label \\ \hline
   $G_{1,1}$ & 140b1 \\ 
   $G_{1,4}$ & 17a1 \\
   $G_{2,1}$ & 17a2 \\
   $G_{1,7}$ & 26b1 \\ \hline
\end{tabular} \hspace{3pt}
\begin{tabular}{|c|c|} \hline
   $E(\mbb{Q}_5)$ & Label \\ \hline
   $G_{1,6}$ & 14a1 \\ 
    &  \\
    & \\
    & \\ \hline
\end{tabular} \hspace{3pt}
\begin{tabular}{|c|c|} \hline
   $E(\mbb{Q}_7)$ & Label \\ \hline
   $G_{1,8}$ & 15a4 \\ 
   $G_{2,2}$ & 15a1 \\
    & \\
    & \\ \hline
\end{tabular}
\caption{Examples of $E(\mbb{Q}_p)_{\mrm{tor}}$ 
with good supersingular reduction}\label{table:ss}
\end{center}
\end{table}

\begin{table}[H]
\begin{center}
\begin{tabular}{|c|c|} \hline
   $E(\mbb{Q}_2(\mu_{2^{\infty}}))$ & Label \\ \hline
   $G_{1,4}$ & 33a3 \\
   $G_{1,8}$ & 15a5 \\
   $G_{2,1}$ & 33a1 \\
   $G_{2,4}$ & 15a2 \\
   $G_{4,1}$ & 15a1 \\ 
   & \\ \hline 
   $E(\mbb{Q}_3(\mu_{3^{\infty}}))$ & Label \\ \hline
   $G_{1,2}$ & 56b1 \\
   $G_{1,3}$ & 26a2 \\ 
   $G_{1,5}$ & 11a1 \\
   $G_{1,6}$ & 14a3 \\
   $G_{3,1}$ & 26a1 \\
   $G_{3,2}$ & 14a1 \\ 
   & \\ \hline
\end{tabular} 
\begin{tabular}{|c|c|} \hline
   $E(\mbb{Q}_5(\mu_{5^{\infty}}))$ & Label \\ \hline
   $G_{1,2}$ & 46a1 \\
   $G_{1,3}$ & 19a1 \\ 
   $G_{1,4}$ & 39a2 \\
   $G_{1,5}$ & 11a2 \\
   $G_{1,7}$ & 26b1 \\
   $G_{1,8}$ & 17a3 \\ 
   $G_{1,9}$ & 26a1 \\
   $G_{1,10}$ & 38b2 \\
   $G_{2,1}$ & 39a1 \\
   $G_{2,2}$ & 17a1 \\
   $G_{5,1}$ & 11a1 \\
   $G_{5,2}$ & 38b1 \\
   & \\
   & \\
   \hline
\end{tabular} 
\begin{tabular}{|c|c|} \hline
   $E(\mbb{Q}_7(\mu_{7^{\infty}}))$ & Label \\ \hline 
   $G_{1,3}$ & 104a1 \\ 
   $G_{1,4}$ & 17a1 \\ 
   $G_{1,5}$ & 38b1 \\ 
   $G_{1,6}$ & 20a1 \\ 
   $G_{1,7}$ & 26b2 \\ 
   $G_{1,9}$ & 19a2 \\
   $G_{1,10}$ & 11a1 \\
   $G_{1,11}$ & 75a1 \\
   $G_{1,12}$ & 30a1 \\
   $G_{1,13}$ & 57a1 \\
   $G_{2,1}$ & 17a2 \\ 
   $G_{2,3}$ & 30a2 \\
   $G_{3,1}$ & 19a1 \\ 
   $G_{7,1}$ & 26b1 \\
   \hline
\end{tabular}  
\caption{Examples of $E(\mbb{Q}_p(\mu_{p^{\infty}}))_{\mrm{tor}}$ 
with good ordinary reduction}\label{table:cyclo:ord}
\end{center}
\end{table}
\begin{table}[H]
\begin{center}
\begin{tabular}{|c|c|} \hline
   $E(\mbb{Q}_2(\mu_4))$ & Label \\ \hline
   $G_{1,4}$ & 15a5 \\
   $G_{2,1}$ & 33a1 \\
   $G_{2,2}$ & 15a2 \\
   $G_{2,4}$ & 15a4 \\
   $G_{4,1}$ & 15a1 \\ \hline 
\end{tabular} 
\caption{Examples of $E(\mbb{Q}_2(\mu_4))_{\mrm{tor}}$ 
with good ordinary reduction}\label{table:cyclo:p=2}
\end{center}
\end{table}
\begin{table}[H]
\begin{center}
\begin{tabular}{|c|c|} \hline
   $E(\mbb{Q}_2(\mu_{2^{\infty}}))$ & Label \\ \hline
   $G_{1,1}$ & 67a1 \\
   $G_{1,3}$ & 19a1 \\ 
   $G_{1,5}$ & 11a1 \\ 
   & \\ \hline
   $E(\mbb{Q}_3(\mu_{3^{\infty}}))$ & Label \\ \hline
   $G_{1,1}$ & 140b1 \\
   $G_{1,4}$ & 17a1 \\
   $G_{2,1}$ & 17a2 \\
   $G_{1,7}$ & 26b1 \\ \hline
\end{tabular}
\begin{tabular}{|c|c|} \hline
   $E(\mbb{Q}_5(\mu_{5^{\infty}}))$ & Label \\ \hline
   $G_{1,6}$ & 14a1 \\
    & \\
    & \\
    & \\ \hline
   $E(\mbb{Q}_7(\mu_{7^{\infty}}))$ & Label \\ \hline
   $G_{1,8}$ & 15a4 \\
   $G_{2,2}$ & 15a1 \\
    & \\
    & \\ \hline
\end{tabular}
\caption{Examples of $E(\mbb{Q}_p(\mu_{p^{\infty}}))_{\mrm{tor}}$ 
with good supersingular reduction}\label{table:cyclo:ss}
\end{center}
\end{table}

\bibliographystyle{alpha}
\bibliography{references}

\end{document}